\documentclass[11pt]{article}
\usepackage{amsmath,amssymb,amscd,color}
\newcommand{\ep}{{ \epsilon  }}

\newcommand{\bq}{\begin{equation}}
\newcommand{\eq}{\end{equation}}
\newcommand{\pa}{\partial}
\newcommand{\R}{{ \mathbb{R}  }}

\newcommand{\bbr}{{ \mathbb{R}  }}

\newcommand{\calH}{{ \mathcal H (\bbr^3) }}
\newcommand{\calV}{{ \mathcal V }}

\newcommand{\calVs}{{ \calV_{\sigma}  }}

\newcommand{\ify}{\infty}
\newcommand{\bke}[1]{\left( #1 \right)}

\newcommand{\norm}[1]{\left\Vert #1 \right\Vert}
\newcommand{\abs}[1]{\left| #1 \right|}

\newcommand{\na}{\nabla}

\newcommand{\ddt}{\frac{d}{dt}}

\newcommand{\wx}{\langle x \rangle}
\newcommand{\Del}{\Delta}

\newcommand{\ltd}{{L^2(\bbr^3)}}
\newcommand{\lod}{{L^1(\bbr^3)}}
\newcommand{\td}{{L^2(0,T;\ltd)}}

\newcommand{\pq}[2]{{L^{#1}(0,T; L^{#2}(\bbr^3))}}

\newcommand{\intd}{\int_{\bbr^3}}
\newcommand{\qed}{\hfill\fbox{}\par\vspace{.2cm}}
\begin{document}
\bibliographystyle{plain}
%\english

% \newtheorem{thm}{Theorem}[section]
% \newtheorem{cor}[thm]{Corollary}
% \newtheorem{lem}[thm]{Lemma}
% \newtheorem{prop}[thm]{Proposition}
% \theoremstyle{definition}
% \newtheorem{defn}[thm]{Definition}
% \theoremstyle{remark}
% \newtheorem{rem}[thm]{Remark}

\newtheorem{defn}{Definition}
\newtheorem{lemma}[defn]{Lemma}
\newtheorem{proposition}{Proposition}
\newtheorem{theorem}[defn]{Theorem}
\newtheorem{cor}{Corollary}
\newtheorem{remark}{Remark}
\numberwithin{equation}{section}

\def\Xint#1{\mathchoice
   {\XXint\displaystyle\textstyle{#1}}%
   {\XXint\textstyle\scriptstyle{#1}}%
   {\XXint\scriptstyle\scriptscriptstyle{#1}}%
   {\XXint\scriptscriptstyle\scriptscriptstyle{#1}}%
   \!\int}
\def\XXint#1#2#3{{\setbox0=\hbox{$#1{#2#3}{\int}$}
     \vcenter{\hbox{$#2#3$}}\kern-.5\wd0}}
\def\ddashint{\Xint=}
\def\dashint{\Xint-}
\def\aint{\Xint\diagup}

\newenvironment{proof}{{\bf Proof.}}{\hfill\fbox{}\par\vspace{.2cm}}
\newenvironment{pfthm1}{{\par\noindent\bf
            Proof of Theorem 1. }}{\hfill\fbox{}\par\vspace{.2cm}}
\newenvironment{pfthm2}{{\par\noindent\bf
            Proof of Theorem 2. }}{\hfill\fbox{}\par\vspace{.2cm}}
\newenvironment{pfthm3}{{\par\noindent\bf
Proof of Theorem 3. }}{\hfill\fbox{}\par\vspace{.2cm}}
\newenvironment{pfthm4}{{\par\noindent\bf
Proof of Theorem 4. }}{\hfill\fbox{}\par\vspace{.2cm}}
\newenvironment{pfthm5}{{\par\noindent\bf
Proof of Theorem 5. }}{\hfill\fbox{}\par\vspace{.2cm}}
\newenvironment{pflemsregular}{{\par\noindent\bf
            Proof of Lemma \ref{sregular}. }}{\hfill\fbox{}\par\vspace{.2cm}}

\title{Existence of Smooth Solutions to Coupled Chemotaxis-Fluid Equations}
\author{Myeongju Chae, Kyungkeun Kang and Jihoon Lee}

\date{}

\maketitle
\begin{abstract}
  We consider a system coupling the parabolic-parabolic Keller-Segel equations to the
  incompressible Navier-Stokes equations in spatial dimensions two and three.
  We establish the local existence of regular solutions and
  present some blow-up criterions. For two
  dimensional Navier-Stokes-Keller-Segel equations, regular solutions constructed locally in time
  are, in reality, extended globally under some assumptions pertinent to
  experimental observation in \cite{TCDWKG} on the consumption rate and chemotactic sensitivity.
  We also show the existence of global weak solutions in spatially three dimensions
  with stronger restriction on the consumption rate and chemotactic sensitivity.
  \newline{\bf 2000 AMS Subject
Classification}: 35Q30, 35Q35, 76Dxx, 76Bxx
\newline {\bf Keywords}: chemotaxis-fluid equations, Keller-Segel,
Navier-Stokes system, global solutions, energy estimates
\end{abstract}

\section{Introduction}
 \setcounter{equation}{0}
 In this paper, we consider mathematical models  describing the
 dynamics of  oxygen diffusion and consumption, chemotaxis, and
 viscous incompressible fluids in
$\bbr^d$, with $d=2,\,3$. Bacteria or microorganisms often live in
fluid, in which the biology of chemotaxis is intimately related to
the surrounding physics.
%In particular, the collective behavior of the
%oxygen-driven swimming bacteria in an aquatic fluid will be
%considered.
Such a model was proposed by Tuval et al.\cite{TCDWKG} to describe
the dynamics of swimming bacteria, {\it Bacillus subtilis}. We
consider the following equations in \cite{TCDWKG} and set $Q_{T} =
(0,\, T] \times \R^{d}$ with $d=2,\, 3$:
 \begin{equation}\label{KSNS} \left\{
\begin{array}{ll}
\partial_t n + u \cdot \nabla  n - \Delta n= -\nabla\cdot (\chi (c) n \nabla c),\\
\vspace{-3mm}\\
\partial_t c + u \cdot \nabla c-\Delta c =-k(c) n,\\
\vspace{-3mm}\\
\partial_t u + u\cdot \nabla u -\Delta u +\nabla p=-n \nabla
\phi,\quad
\nabla \cdot u=0 %\qquad t \in (0,\, T], \qquad x \in \R^d
\end{array}
\right. \quad\mbox{ in }\,\, (x,t)\in \R^d\times (0,\, T],
\end{equation}
where $c(t,\,x) : Q_{T} \rightarrow \R^{+}$, $n(t,\,x) : Q_{T}
\rightarrow \R^{+}$, $u(t,\, x) : Q_{T} \rightarrow \R^{d}$ and
$p(t,x) :  Q_{T} \rightarrow \R$ denote the oxygen concentration,
cell concentration, fluid velocity, and scalar pressure,
respectively. The nonnegative function $k(c)$ denotes the oxygen
consumption rate, and the nonnegative function $\chi (c)$ denotes
chemotactic sensitivity. Initial data are given by $(n_0(x), c_0(x),
u_0(x))$. To describe the fluid motions, we use Boussinesq
approximation to denote the effect due to heavy bacteria. The
time-independent function $\phi =\phi (x)$ denotes the potential
function produced by different physical mechanisms, e.g., the
gravitational force or centrifugal force. Thus, $\phi(x)=ax_{d}$ is
one example of gravity force, and $\phi(x)=\phi(|x|)$ is an example
of  centrifugal force. Experiments in \cite{TCDWKG} suggest that the
functions $k(c)$ and $\chi (c)$ are
 constants at large $c$ and rapidly approach zero below some critical
$c^{*}$. Hence, in \cite{TCDWKG}, these functions are approximated
by step functions, e.g., $k(c)=\kappa_1 \theta (c-c^{*})$ and $\chi
(c)= \kappa_2 \theta(c-c^{*})$ for some positive constants
$\kappa_1$ and $\kappa_2$. Also  in \cite{CFKLM}, numerical
simulation of plumes was obtained for the same species of bacteria
in \cite{TCDWKG} in two dimensions. Furthermore, they assumed that
the functions $\chi(c)$ and $k(c)$ are constant multiples of each
other, i.e., $\chi(c)=\mu k(c)$.
\\
The main goals of this paper are to show the local existence of
smooth solutions in two and three dimensions with the general
condition on the oxygen consumption rate and chemotactic
sensitivity, and to demonstrate global existence of smooth solutions
in two dimensions  and weak solutions in three dimensions with
appropriate assumptions of $\chi(c)$, $k(c)$, $\phi$ and initial
data. Here we mention the related works for the result in this
paper. If we ignore the coupling of the fluids, we obtain the
angiogenesis type system. The classical model to describe the motion
of cells was suggested by Patlak\cite{Patlak} and
Keller-Segel\cite{KS1, KS2}. It consists of a system of the dynamics
of  cell density $n=n(t,x)$ and the concentration of chemical
attractant substance $c=c(t,x)$ and is given as
\begin{equation}\label{KS-nD} \,\,\left\{
 \begin{array}{c}
 n_t=\Delta n-\nabla \cdot(n \chi\nabla c),\\
 \vspace{-3mm}\\
 \alpha c_t=\Delta c-\tau c+n,\\
 \vspace{-3mm}\\
 n(x,\,0)=n_0(x),\quad c(x,\,0)=c_0(x),
 \end{array}
 \right.
\end{equation}
where $\chi$ is the sensitivity and $\tau^{-\frac12}$ represents the
activation length. The system in \eqref{KS-nD} has been extensively
studied by many authors(see \cite{Her-Vela, Horst-Wang, NSY, OY,
Win} and references therein). For the chemical consumption model by the cell or bacteria, we refer to the following chemotaxis model
motivated by angiogenesis.
\begin{equation}\label{angio} \left\{
 \begin{array}{c}
 n_t=\Delta n-\nabla \cdot(n\chi(c) \nabla c),\\
 \vspace{-3mm}\\
 c_t=-c^m n,\\
 \vspace{-3mm}\\
 n(x,\,0)=n_0(x),\quad c(x,\,0)=c_0(x).
 \end{array}
 \right.
\end{equation}
The global existence of the weak solution to the system in
\eqref{angio} was obtained by Corrias, Perthame and Zaag\cite{CPZ1,
CPZ2} with a small data assumption of $\| n_0\|_{L^{\frac{d}{2}}}$.
The bacterial movement toward the concentration gradient model in
the absence of the fluid, i.e., $u=0$, was recently studied. When $u
\equiv 0$, $\chi(c)\equiv \chi$  and $k(c)\equiv c$ in \eqref{KSNS},
it was shown in \cite{YTao} that there exists a uniquely global
bounded solution if $\displaystyle 0 < \chi \le
\frac{1}{6(d+1)\|c_0\|_{L^{\infty}}}$.\\
If the flow of the fluid is slow, then the Navier-Stokes equations
can be simplified to the Stokes equations. For the case
$\chi(c)\equiv \chi$, Lorz\cite{Lorz} showed the local existence of
the solution for the Keller-Segel-Stokes system. In
two dimensions, Duan, Lorz, and Markowich\cite{DLM} showed the
global existence of the weak solution to the Keller-Segel-Stokes
equations with the small data assumptions on $c_0$, $\phi$ and the
assumptions on the functions such that
\begin{equation}\label{Assumption1}
\chi(c)>0,\quad \chi'(c)\ge 0,\quad k'(c) >0,\quad \frac{d^2}{dc^2}
\left(\frac{k(c)}{\chi(c)} \right)<0.
\end{equation}
In \cite{Liu-Lorz}, Liu and Lorz showed the global existence of a
weak solution to the two-dimensional Keller-Segel-Navier-Stokes
equations with similar assumptions on $k$ and $\chi$ to those in
\eqref{Assumption1}. The equation of $n$ in \eqref{KSNS} could have
been replaced by a porous medium equation, i.e., $\Delta n$ is
replaced by $\Delta n^m$ and the following Keller-Segel-Stokes
system has been considered in \cite{FLM}.
 \begin{equation}\label{KS-P-S} \left\{
\begin{array}{ll}
\partial_t n + u \cdot \nabla  n - \Delta n^m= -\nabla\cdot (\chi (c) n \nabla c),\\
\vspace{-3mm}\\
\partial_t c + u \cdot \nabla c-\Delta c =-k(c) n,\\
\vspace{-3mm}\\
\partial_t u  -\Delta u +\nabla p=-n \nabla \phi,\quad
\nabla \cdot u=0, %\qquad t \in (0,\, T], \qquad x \in \Omega
\end{array}
\right. \quad\mbox{ in }\,\, (x,t)\in \R^d\times (0,\, T].
\end{equation}
In \cite{FLM}, Francesco, Lorz and Markowich showed the global
existence of the bounded solution to \eqref{KS-P-S} when $m \in
(\frac32,\, 2]$. In \cite{Liu-Lorz}, Liu and Lorz proved the global
existence of the weak solution when $m >\frac43$ in three
dimensions. For the Keller-Segel-Navier-Stokes system \eqref{KSNS},
Duan, Lorz and Markowich\cite{DLM} showed the global-in-time
existence of the $H^3(\R^d)$-solution, near constant states, to
\eqref{KSNS}, i.e., if the initial data $\|(n_0-n_{\infty},\, c_0,\,
u_0)\|_{H^3}$ is sufficiently small, then there exists a unique
global solution.

As mentioned earlier, the aim of this paper is to obtain the
local-in-time existence of the smooth solution in two and three
dimensions and the global-in-time existence of the classical
solution to \eqref{KSNS} in two dimensions under the minimal
assumptions on the consumption rate and chemotactic sensitivity. Now
we are ready to state our main results. The first result in this
article is the local existence in time of the smooth solutions to
\eqref{KSNS}. Comparing with the result in \cite{DLM}, we show the
local-in-time existence without smallness of the initial data.

\begin{theorem}\label{Theorem2} (Local existence)\,\,
Let $m\geq 3$ and $d=2,3$. Assume that $\chi$, $k\in C^m(\R^+)$ and $k(0)=0$, $\|
\nabla^l \phi \|_{L^{\infty}}<\infty$ for $1\le |l|\le m$. There
exists $T>0$, the maximal time of existence, such that, if the
initial data $(n_0,\,c_0,\, u_0)\in H^{m-1}(\R^d) \times H^{m}(\R^d)
\times H^{m}(\R^d)$, then there exists a unique classical solution
$(n,\,c,\,u)$ of \eqref{KSNS} satisfying for any $t<T$
\[
(n,c,u) \in L^{\infty}(0,\, t; H^{m-1}(\R^d)\times H^{m}(\R^d)
\times H^{m}(\R^d)),
\]
\[
(\nabla n,\, \nabla c,\, \nabla u) \in L^2(0,\, t;
H^{m-1}(\R^d)\times H^{m} (\R^d)\times H^{m}(\R^d)).
\]
\end{theorem}

\begin{remark}
For simplicity, we denote
\[
\|(n(t), c(t), u(t))\|_{X_m}:= \|n (t)\|_{ H^{m-1}(\R^d)}+\|c
(t)\|_{ H^{m}(\R^d)}+\|u (t)\|_{ H^{m}(\R^d)}.
\]
We remark that if $T$ is the maximal time of existence with
$T<\infty$ in Theorem \ref{Theorem2}, then
\[
\limsup_{t\nearrow T} \|(n(t), c(t), u(t))\|_{X_m}^2 +\int_0^T
\|(n(t), c(t), u(t))\|^2_{X_{m+1}}=\infty,\qquad m\geq 3.
\]
\end{remark}

Secondly, we obtain two blow-up criteria for the system \eqref{KSNS}
depending on dimensions.
\begin{theorem}\label{Theorem3}
Suppose that $\chi$, $k$, $\phi$ and the initial data $(n_0, c_0, u_0)$
satisfy all the assumptions presented in Theorem \ref{Theorem2}. If
$T^*$, the maximal time existence in Theorem \ref{Theorem2}, is
finite, then one of the following  is true in each case of two or three dimensions, respectively:
\begin{equation}\label{2d-reg}
(2D)\qquad \int_0^{T^*}\norm{\nabla c}^2_{L^{\infty}(\R^2)}=\infty.
\end{equation}
\begin{equation}\label{3d-nse-reg}
(3D)\qquad
\int_0^{T^*}\norm{u}^q_{L^{p}(\R^3)}+\int_0^{T^*}\norm{\nabla
c}^2_{L^{\infty}(\R^3)}=\infty, \qquad
\frac{3}{p}+\frac{2}{q}=1,\,\,3<p\leq \infty.
\end{equation}
\end{theorem}

\begin{remark}
Theorem \ref{Theorem3} can be interpreted as follows: If
$\int_0^T\norm{\nabla c}^2_{L^{\infty}}<\infty$ in two dimensions or
if $\int_0^T (\norm{u}^q_{L^{p}} +\| \na c \|_{L^{\infty}}^2) dt<
\infty$ in three dimensions, then the local solution can persist
beyond time $T$, i.e., $(n,\, c,\, u)\in L^{\infty}(0,\, T+\delta;
X_m)\cap L^2(0,\, T+\delta; X_{m+1})$ for some $\delta>0$.
\end{remark}

The third main result is  the global existence of the smooth solutions
in the  two-dimensional spatial domain $\R^2$.
Motivated by
experiments in \cite{TCDWKG} and \cite{CFKLM}, we assume that the oxygen
consumption rate $k(c)$ and chemotactic sensitivity $\chi(c)$
satisfy the following conditions: \\
\\
{\bf{(A)}} There exists a constant $\mu$ such that $\sup | \chi (c)-
\mu k(c)| < \epsilon$ for a sufficiently
small $\epsilon>0$.\vspace{0.2cm}\\
{\bf{(B)}} $\chi(c), k(c), \chi'(c), k'(c)$ are all non-negative,
i.e.,
$\chi(c),\, k(c),\, \chi'(c),\, k'(c)\geq 0$.\\

We remark that assumption {\bf{(A)}} plays a crucial role in
obtaining $LlogL \times H^1 \times L^2$ type estimates.
\begin{theorem}\label{Theorem4}(Global existence in two dimensions)
Let $d=2$. Suppose that $\chi$, $k$, $\phi$ and the initial data
$(n_0, c_0, u_0)$ satisfy all the assumptions presented in Theorem
\ref{Theorem2}. Assume further that $\chi$ and $k$  satisfy the
assumptions {\bf{(A)}} and {\bf{(B)}} and $\phi \geq 0$. Then the unique regular
solution $(n,\,c, u)$ exists globally in time and satisfies for any
$T<\infty$
\[
(n,\,c,\, u) \in L^{\infty}(0,\, T; H^{m-1}(\R^2)\times
H^m(\R^2)\times H^m(\R^2)),
\]
\[
(\nabla n,\,\nabla c,\,\nabla u) \in L^2(0,\, T; H^{m-1}(\R^2)
\times H^{m}(\R^2)\times H^{m}(\R^2)).
\]
\end{theorem}
\begin{remark}
If we approximate Heaviside functions using the smooth functions, then
 the experiments in \cite{TCDWKG} satisfy the assumptions
{\bf{(A)}} and {\bf{(B)}}. Furthermore, the assumptions on $\phi$
are satisfied by gravitational and centrifugal forces. Also we note
that 2D numerical studies were performed under the assumption $\chi(c)=
\mu k(c)$ in \cite{CFKLM}.
\end{remark}
Our final main theorem is on the global-in-time existence of weak
solutions in three dimensions. The notion of a weak solution of
\eqref{KSNS} is detailed in section 4 (see Definition
\ref{defweak}). For existence of global weak solution, we need
similar restrictions on $k(c)$ and $\chi(c)$ as in Theorem
\ref{Theorem4}. More precisely, compared to {\bf{(A)}}, we
impose a slightly stronger assumption, denoted by {\bf{(AA)}}, which is given as follows:\\
\\
{\bf{(AA)}} There exists a constant $\mu$ such that $\chi(c)-\mu
k(c) =0$.
\\
\\
We are ready to state our last main result.
\begin{theorem}\label{Theorem5}
Let $d=3$ and $(n_0, c_0, u_0)$ satisfy
\begin{align} \label{weakdata} \begin{aligned}
&  n_0 \in L^1(\bbr^3), \quad c_0 \in H^1(\R^3)\cap L^{\infty}(\R^3),\quad u_0 \in \calH,\\
&n_0\ge 0,\quad c_0 \ge 0,\quad \intd  (1+ |\ln n_0| + |x| )n_0 dx
< \infty.
\end{aligned}
\end{align}
 Assume further that  $\chi$ and $k$
$\in C^1(\bbr^+)$ satisfy the assumption {\bf{(AA)}}, {\bf{(B)}} and $k(0)=0$ and $\phi$ satisfies $\phi \geq 0$ and $\| \nabla^{l} \phi \|_{L^{\infty}}< \infty$ for $0\leq |l| \leq 2$.
Then a weak solution $(n,c,u)$ exists globally in time.
%there exists a global weak solution $(n,c,u)$ in the sense of
%Definition \ref{defweak}.
\end{theorem}

The rest of this paper is organized as follows. In Section 2, we
prove local-in-time existence of the smooth solution for the two and
three dimensional chemotaxis system with incompressible
Navier-Stokes equations  and obtain some blow-up criteria for the
solution. In Section 3, we show the global in time existence of the
regular solution in two dimensions. In Section 4, we establish the
existence of a weak solution in three dimensions.

%%%%%%%%%%%%%%%%%%%%%%%%%%%%%%%%%%%%%%%%%%%%%%%%%%%%%%%%%%%%%%%%%%%%%%%%%%%%%%%%%%%
%%%%%%%%%%%%%%%%%%%%%%%%%%%%%%%%%%%%%%%%%%%%%%%%%%%%%%%%%%%%%%%%%%%%%%%%%%%%%%%%%%%
%%%%%%%%%%%%%%%%%%%%%%%%%%%%%%%%%%%%%%%%%%%%%%%%%%%%%%%%%%%%%%%%%%%%%%%%%%%%%%%%%%%
\section{Local existence and blow-up criterion}

\subsection{Local existence} We first consider the chemotaxis system coupled with the
Navier-Stokes equations in two and three dimensions.
We show the local existence of solutions $(n,\, c,\, u)$ in $H^{m-1} \times
H^{m} \times H^{m}$ space with $m \ge 3$.
 \begin{pfthm1}
We construct the solution sequence $(n^{j},\, c^{j}, \, u^{j})_{j
\geq 0}$ by iteratively solving the Cauchy problems on the following
linear equations
\begin{equation}\label{linear-eq-1}
\left\{\begin{array}{ll}  \partial_t n^{j+1}+u^{j} \cdot \nabla
n^{j+1} = \Delta n^{j+1} - \nabla \cdot ( \chi(c^{j}) n^{j+1} \nabla
c^{j}),\\
\vspace{-3mm}\\
\partial_t c^{j+1} +u^{j} \cdot \nabla c^{j+1} =\Delta c^{j+1}
-k(c^{j}) n^{j},\\
\vspace{-3mm}\\
\partial_t u^{j+1} +u^{j} \cdot \nabla u^{j+1} +\nabla p^{j+1}
=\Delta u^{j+1} -n^{j} \nabla \phi,\qquad \nabla \cdot u^{j+1}=0.
\end{array}\right.
\end{equation}
%In the above, we put the nonnegative viscosity coefficient $\nu$ including the case $\nu=0$.

We first set $(n^0(x,t),\, c^0(x,t),\, u^0(x,t))=(n_0(x),\,
c_0(x),\, u_0(x))$. Then, using the same initial data to solve the
linear Stokes type equations and the linear parabolic equations, we
obtain $u^1 (x,t)$, $c^1(x,t)$ and $n^1 (x,t)$, respectively.
%Afterward, we can obtain $n^1 (x,t)$ by solving the linear parabolic equation.
Similarly, we define $(n^j(x,t),\, c^j(x,t),\, u^j(x,t))$
iteratively. For this, we presume that $c^j$ and $n^j$ are
nonnegative and show the existence and the convergence of
solutions in the adequate function spaces. We show the nonnegativity
of $c^j$ and $n^j$ at the end of the proof.

To prove the conclusion, i.e., to obtain contraction in adequate
function spaces, we show the uniform boundness of the sequence of
functions under our
construction via energy estimates.\\
$\bullet$\, (Uniform boundedness)\,\, We here show that the
iterative sequences $(n^j,\, c^j, u^j)$ are in $X_m :=H^{m-1}\times
H^{m} \times H^{m}$ space for all $j\geq 0$. Observing that
\[
\sum_{ |\alpha| \le m-1}\| \partial^{\alpha} (u^{j}
 n^{j+1})
\|_{L^2}
\le C (\|u^j \|_{L^{\infty}} \|
n^{j+1}\|_{H^{m-1}}+\|  n^{j+1}\|_{L^{\infty}} \|
u^{j}\|_{H^{m-1}}),
\]
\[
\norm{ \chi(c^{j}) n^{j+1} \nabla c^{j} }_{H^{m-1}} \leq C \norm{
n^{j+1}}_{H^{m-1}} \left( 1+ \norm{c^{j}}_{H^{m}}^{m}   \right),
\]
\[
\norm{ k(c^{j}) n^{j} }_{H^{m-1}} \leq C \norm{n^{j}}_{H^{m-1}} \left(1+ \norm{c^{j}}_{H^{m}}^{m-1}\right),
\]
we have
the following energy estimates:\\
$(i)$\,\, The estimate of $n^{j+1}$
\[
\frac12 \frac{d}{dt} \norm{n^{j+1}}_{H^{m-1}}^2 + \norm{ \nabla
n^{j+1}}_{H^{m-1}}^2 \leq C \norm{  u^{j} }_{L^{\infty}} \norm{
n^{j+1} }_{H^{m-1}}\norm{\na n^{j+1}}_{H^{m-1}}
\]
\[
+ C \norm{u^{j}}_{H^{m-1}} \norm{n^{j+1}}_{H^{m-1}}\norm{\na
n^{j+1}}_{H^{m-1}} +C \left( 1+ \norm{c^{j}}_{H^{m}}^{m}
\right)\norm{ n^{j+1}}_{H^{m-1}}\norm {\nabla n^{j+1}}_{H^{m-1}}
\]
\begin{equation}\label{eq-n-j-1}
\leq C\left(1+ \norm{u^j}_{H^{m}}^2+\norm{c^{j}}_{H^m}^{2m}
\right)\norm{n^{j+1}}_{H^{m-1}}^2 + \frac12 \norm{ \na
n^{j+1}}_{H^{m-1}}^2.
\end{equation}
$(ii)$\,\, The estimate of $c^{j+1}$
\[
\frac12 \frac{d}{dt} \norm{c^{j+1}}_{H^{m}}^2 + \norm{\nabla
c^{j+1}}_{H^{m}}^2 \leq C \norm { u^{j} }_{L^{\infty}} \norm{c^{j+1}
}_{H^{m}}\norm{\nabla c^{j+1}}_{H^{m}}
\]
\[
+C \norm {u^j }_{H^{m}} \norm {c^{j+1}}_{H^{m}}\norm{\nabla
c^{j+1}}_{H^{m}} +C  \left(1+
\norm{c^{j}}_{H^{m}}^{m-1}\right)\norm{n^{j}}_{H^{m-1}}\norm{ \nabla
c^{j+1} }_{H^{m}}
\]
\begin{equation}\label{eq-c-j-1}
\leq C\left(1+ \norm{c^{j}}_{H^m}^{2(m-1)}
\right)\norm{n^{j}}_{H^{m-1}}^2 +
C\norm{u^{j}}_{H^m}^2\norm{c^{j+1}}_{H^m}^2 + \frac12 \norm{ \na
c^{j+1}}_{H^m}^2.
\end{equation}
$(iii)$\,\, The estimate of $u^{j+1}$
\[
\frac12 \frac{d}{dt} \norm{ u^{j+1}}_{H^{m}}^2 +\norm{\nabla
u^{j+1}}_{H^{m}}^2 \leq C \norm {\nabla u^j}_{L^{\infty}} \norm
{u^{j+1}}_{H^{m}}^2
\]
\[
+C \norm{u^j}_{H^{m}}\norm{\nabla u^{j+1}}_{L^{\infty}}
\norm{u^{j+1}}_{H^{m}} +C \norm{n^j}_{H^{m-1}} \norm{\na
u^{j+1}}_{H^{m}}
\]
\begin{equation}\label{eq-u-j-1}
\leq C \norm{u^{j}}_{H^m}^2 \norm{u^{j+1}}_{H^m}^2 + C
\norm{n^{j}}_{H^{m-1}}^2+ \frac12 \norm{ \na u^{j+1}}_{H^m}^2,
\end{equation}
where standard commutator estimates are used.
%In \eqref{eq-c-j-1}, we used
%\[
%\sum_{|\alpha| \leq m }\int \partial^{\alpha} (u^j \cdot \nabla c^{j+1} ) \partial^{\alpha}c^{j+1} = \sum_{ 1 \leq |\beta|, \beta < \alpha} \int \partial^{\beta} u^j\cdot \nabla \partial^{\alpha} c^{j+1} \partial^{\alpha-\beta} c^{j+1}+ \int \partial^{\alpha} u^j \cdot \nabla c^{j+1} \partial^{\alpha}c^{j+1}
%\]
%\[
%= \sum_{1 \leq |\beta|, \beta < \alpha} \int \partial^{\beta} u^{j} \cdot \nabla \partial^{\alpha} c^{j+1} \partial^{\alpha- \beta} c^{j+1}
%- \sum_{ |\gamma|=1} \int \partial^{\alpha- \gamma}u^{j} \cdot
%\nabla \partial^{\gamma} c^{j+1} \partial^{\alpha} c^{j+1} - \int
%\partial^{\alpha-\gamma} u^{j}\cdot \nabla c^{j+1}
%\partial^{\alpha+\gamma}c^{j+1}.
%\]
%Hence,
%\[
%\sum_{|\alpha| \leq m }\left|\int \partial^{\alpha} (u^j \cdot
%\nabla c^{j+1} )  \partial^{\alpha}c^{j+1}\right| \leq C\norm {u^j
%}_{H^{m}} \norm {c^{j+1}}_{H^{m}}\norm{\nabla c^{j+1}}_{H^{m}}.
%\]
We show that there exists a constant $M>0$ such that, for any
$j$, the following inequality holds for a small time interval $[0,\,
T]$ ($T$ will be specified later):
\[
\sup_{ 0 \leq t \leq T}\left( \norm{ n^{j}}_{H^{m-1}}^2
+\norm{c^{j}}_{H^{m}}^2 +\norm{u^{j}}_{H^m}^2  \right)
\]
\begin{equation}\label{ineq-uniform}
%\sup_{ 0 \leq t \leq T}\left( \norm{ n^{j}}_{H^{m-1}}^2
%+\norm{c^{j}}_{H^{m}}^2 +\norm{u^{j}}_{H^m}^2  \right)
+ \int_0^T\norm{\na n^{j}}_{H^{m-1}}^2 + \norm{ \na c^j}_{H^{m}}^2
+\norm{\na u^j}_{H^m}^2 dt \leq M.
\end{equation}
Here $M$ is a number with $M\geq 4(\norm{ n_0}_{H^{m-1}}^2
+\norm{c_0}_{H^{m}}^2 +\norm{u_0}_{H^m}^2)$.

We prove \eqref{ineq-uniform} via an inductive argument. Suppose
\eqref{ineq-uniform} hold for $j \leq i$.
%We then show the $i+1$-th step.
If we add \eqref{eq-n-j-1}, \eqref{eq-c-j-1}, and
\eqref{eq-u-j-1} and use Young's inequality, then we have
\[
\frac{d}{dt} ( \norm{n^{i+1}}_{H^{m-1}}^2 + \norm{c^{i+1}}_{H^{m}}^2
+ \norm{u^{i+1}}_{H^{m}}^2) + \norm{\na n^{i+1}}_{H^{m-1}}^2
+ \norm{\na c^{i+1}}_{H^{m}}^2+ \norm{ \na u^{i+1}}_{H^m}^2
\]
\[
\leq C\left(1+ \norm{u^i}_{H^{m}}^2+\norm{c^{i}}_{H^m}^{2m}
\right)\norm{n^{i+1}}_{H^{m-1}}^2 +C\left(1+
\norm{c^{i}}_{H^m}^{2(m-1)}  \right)\norm{n^{i}}_{H^{m-1}}^2
\]
\[
+ C\norm{u^{i}}_{H^m}^2\norm{c^{i+1}}_{H^m}^2+C \norm{u^{i}}_{H^m}^2
\norm{u^{i+1}}_{H^m}^2+ C \norm{n^{i}}_{H^{m-1}}^2
\]
\[
\leq C(1+M+M^{m}) \norm{n^{i+1}}_{H^{m-1}}^2+CM
\norm{c^{i+1}}_{H^m}^2
\]
\[
+ CM\norm{u^{i+1}}_{H^m}^2 +C(1+M^{m-1})M+CM.
\]
In the last inequality, we use the induction hypothesis. Hence,
we get
\[
\frac{d}{dt} ( \norm{n^{i+1}}_{H^{m-1}}^2 + \norm{c^{i+1}}_{H^{m}}^2
+ \norm{u^{i+1}}_{H^{m}}^2)+ \norm{\na n^{i+1}}_{H^{m-1}}^2 +
\norm{\na c^{i+1}}_{H^{m}}^2+ \norm{ \na u^{i+1}}_{H^m}^2
\]
\begin{equation}\label{ineq-induction}
\leq C(1+M+M^{m})( \norm{n^{i+1}}_{H^{m-1}}^2 +
\norm{c^{i+1}}_{H^{m}}^2 + \norm{u^{i+1}}_{H^{m}}^2) +C(1+M^{m-1})M.
\end{equation}
We choose time $T$ such that $\max \{C(1+M +M^{m})T,\,\,
C(1+M^{m-1})T\}\leq \frac14$. Then from Gronwall's inequality,
we have \eqref{ineq-uniform}. In short, we have $(n^{j+1},\,
c^{j+1}, \, u^{j+1}) \in L^{\infty}(0,\,T; X_m)$ and $(\nabla
n^{j+1},\,\nabla c^{j+1}, \, \nabla u^{j+1}) \in L^{2}(0,\,T; X_m)$
and the uniform bound \eqref{ineq-uniform} for small $T>0$.

Also if we multiply $(c^{j+1})^{q-1}$ on the both sides of the
second equation of \eqref{linear-eq-1} and integrate over spatial
varaibles, then we obtain
\[
\frac{1}{q} \frac{d}{dt} \| c^{j+1}\|_{L^q}^q
+\frac{4(q-1)}{q^2} \| \nabla (c^{j+1})^{\frac{q}{2}}\|_{L^2}^2 \le 0.
\]
Thus, the $L^{\infty}$ norm of $c^{j+1}$ is uniformly bounded, which
implies that $\chi(c^{j})$ and $k(c^{j})$ are uniformly
bounded for all $j\geq 0$.\\
\\
$\bullet$ (Contraction) The estimate of this part is similar to that
of the previous one. For convenience, we denote $\delta
f^{j+1}:=f^{j+1}-f^{j}$. Subtracting the $j$-th equations from
the $(j+1)$-th equations, we have the following equations for $\delta n^{j+1}$,
$\delta c^{j+1}$ and $\delta u^{j+1}$ :
\begin{equation}\label{approximat-linear-eq-1}
\left\{\begin{array}{l}
\partial_t\delta  n^{j+1}+u^{j} \cdot
\nabla \delta n^{j+1}-\Delta \delta n^{j+1} =-\delta u^{j} \cdot
\nabla n^{j}-\nabla \cdot(
\chi(c^{j}) \delta n^{j+1} \nabla c^{j})\\
\qquad\qquad\qquad\qquad\qquad\qquad\qquad\,\,\quad-\nabla \cdot (
\chi(c^{j}) n^{j} \nabla c^{j})
+\nabla \cdot ( \chi(c^{j-1}) n^{j} \nabla c^{j-1}),\\
\vspace{-2mm}\\
\partial_t \delta c^{j+1} +u^{j} \cdot \nabla \delta c^{j+1}
-\Delta \delta c^{j+1}=-\delta u^{j} \cdot \nabla c^{j}
-k(c^{j})\delta n^{j}+(k(c^{j})-k(c^{j-1})) n^{j-1},\\
\vspace{-2mm}\\
\partial_t \delta u^{j+1} +u^{j} \cdot \nabla \delta u^{j+1}
+\nabla \delta p^{j+1}
-\Delta \delta u^{j+1}=-\delta u^{j} \cdot \nabla u^{j} -\delta n^{j} \nabla \phi,\\
\nabla \cdot\delta  u^{j+1}=0.
\end{array}\right.
\end{equation}
$(i)$\,\, The estimate of $\delta n^{j+1}$.\\
Using the following standard commutator estimates
\[
\sum_{ |\alpha| \le m-1}\int [\partial^{\alpha} (u^{j} \cdot\nabla
\delta n^{j+1}) \partial^{\alpha}\delta n^{j+1}- (u^{j} \cdot\nabla
\partial^{\alpha}\delta n^{j+1}) \partial^{\alpha}\delta n^{j+1}]
\]
\[
\le C (\| \nabla u^{j} \|_{L^{\infty}} \| \delta
n^{j+1}\|^2_{H^{m-1}}+\| u^{j}\|_{H^{m-1}}\|\nabla\delta
n^{j+1}\|_{L^{\infty}}\| \delta n^{j+1}\|_{H^{m-1}}),
\]
%we estimate $\delta n^{j+1}$ in $H^{m-1}$ space as follows:
we have the following estimate:
\[
\frac12 \frac{d}{dt} \| \delta n^{j+1} \|_{H^{m-1}}^2+ \| \nabla
\delta n^{j+1}\|_{H^{m-1}}^2\le C \| \nabla u^{j} \|_{L^{\infty}} \|
\delta n^{j+1} \|_{H^{m-1}}^2
\]
\[
+C \|  u^{j} \|_{H^{m-1}} \| \delta n^{j+1} \|_{H^{m}}\| \delta
n^{j+1} \|_{H^{m-1}} +C \|  \delta u^{j} \|_{L^{\infty}} \| n^j
\|_{H^{m-1}} \| \nabla \delta n^{j+1}\|_{H^{m-1}}
\]
\[
+C \|  \delta u^j \|_{H^{m-1}} \|  n^{j} \|_{L^{\infty}} \| \nabla
\delta n^{j+1} \|_{H^{m-1}}+ C(\| c^j\|_{H^{m}}+\|
c^j\|_{H^{m}}^{m-1})  \| \delta n^{j+1}\|_{H^{m-1}} \| \nabla \delta
n^{j+1}\|_{H^{m-1}}\]
\[
+C\| n^j \|_{H^{m-1}}(1+ \| c^{j}\|_{H^m}^{m-1}+\|
c^{j-1}\|_{H^m}^{m-1}) \|\delta c^{j} \|_{H^m} \| \nabla \delta
n^{j+1}\|_{H^{m-1}}.
\]
\[
\leq C (\| u^{j} \|_{H^{m}}+\| u^{j} \|^2_{H^{m}}) \| \delta n^{j+1}
\|_{H^{m-1}}^2+C \|  \delta u^j \|^2_{H^{m}}\| n^j \|^2_{H^{m-1}}+
C(\| c^j\|^2_{H^{m}}+\| c^j\|_{H^{m}}^{2(m-1)})  \| \delta
n^{j+1}\|^2_{H^{m-1}}
\]
\[
+C\| n^j \|^2_{H^{m-1}}(1+ \| c^{j}\|_{H^m}^{2(m-1)}+\|
c^{j-1}\|_{H^m}^{2(m-1)}) \|\delta c^{j} \|^2_{H^m} +\frac{1}{2}\|
\nabla \delta n^{j+1}\|_{H^{m-1}}^2.
\]
%Similarly, we can estimate $\delta c^{j+1}$  as
%follows:
$(ii)$\,\, The estimate of $\delta c^{j+1}$.
\[
\frac12 \frac{d}{dt} \| \delta c^{j+1} \|_{H^{m}}^2+ \| \nabla \delta
c^{j+1}\|_{H^{m}}^2
%\]
%\[
\le C \|  u^{j} \|_{L^{\infty}}\| \delta c^{j+1} \|_{H^{m}}\|\nabla
\delta c^{j+1} \|_{H^m}
\]
\[
+C \|  u^j \|_{H^{m-1}} \| \delta c^{j+1} \|_{H^{m}} \| \nabla\delta
c^{j+1}\|_{H^{m}} +C \|  \delta u^j \|_{L^{\infty}} \|  c^j
\|_{H^{m}} \|\nabla \delta c^{j+1}\|_{H^{m}}
\]
\[
+C \|  c^j \|_{H^m} \|  \delta u^j \|_{H^{m-1}} \| \nabla\delta
c^{j+1} \|_{H^{m}}+C(1+ \| c^j \|_{H^m}^{m-1})\| \delta
n^{j}\|_{H^{m-1}}\| \nabla \delta c^{j+1}\|_{H^m}
\]
\[
+C\| n^{j-1}\|_{H^{m-1}}(1+ \|  c^{j}\|_{H^m}^{m-1}+\|
c^{j-1}\|_{H^{m}}^{m-1})
 \| \delta c^{j} \|_{H^{m}} \|\nabla \delta c^{j+1} \|_{H^m}
\]
\[
\leq C \|  u^j \|^2_{H^{m}} \| \delta c^{j+1} \|^2_{H^{m}}+C \|  c^j
\|^2_{H^m} \|  \delta u^j \|^2_{H^{m}}+C(1+ \| c^j
\|_{H^m}^{2(m-1)})\| \delta n^{j}\|^2_{H^{m-1}}
\]
\[
+C\| n^{j-1}\|^2_{H^{m-1}}(1+ \|  c^{j}\|_{H^m}^{2(m-1)}+\|
c^{j-1}\|_{H^{m}}^{2(m-1)})
 \| \delta c^{j} \|^2_{H^{m}}+\frac{1}{2}\| \nabla\delta
 c^{j+1}\|^2_{H^{m}},
\]
where we used the Mean Value Theorem for the last term.\\
%We estimate $\delta u^{j+1}$ as follows:
$(iii)$\,\, The estimate of $\delta u^{j+1}$.
\[
\frac12 \frac{d}{dt} \| \delta u^{j+1} \|_{H^{m}}^2+ \| \nabla
\delta u^{j+1}\|_{H^{m}}^2\le C \| \nabla u^{j} \|_{L^{\infty}}\|
\delta u^{j+1} \|_{H^{m}}^2
\]
\[
+C \|  u^j \|_{H^{m}} \| \nabla \delta u^{j+1} \|_{L^{\infty}} \|
\delta u^{j+1}\|_{H^{m}}+C \| \delta u^j\|_{L^{\infty}} \| \nabla
u^{j}\|_{H^{m-1}} \| \delta u^{j+1}\|_{H^{m+1}}
\]
\[
+C \| \delta u^{j} \|_{H^{m-1}} \| \nabla u^{j} \|_{L^{\infty}} \|
 \delta u^{j+1}\|_{H^{m+1}} +C  \| \delta n^{j} \|_{H^{m-1}} \|  \delta
u^{j+1}\|_{H^{m+1}}
\]
\[
\leq C \|  u^j \|_{H^{m}}\| \delta u^{j+1} \|_{H^{m}}^2+C\|  u^j
\|^2_{H^{m}}\| \delta u^{j} \|_{H^{m}}^2+C  \| \delta n^{j}
\|^2_{H^{m-1}}+\frac{1}{2} \|  \delta u^{j+1}\|^2_{H^{m+1}},
\]
where similar standard commutator estimates are used as in the case
of $\delta n^{j+1}$. Using Young's inequality, we have
\[
\frac{d}{dt}( \| \delta n^{j+1} \|_{H^{m-1}}^2+\| \delta
c^{j+1} \|_{H^m}^2 +\| \delta u^{j+1} \|_{H^m}^2   ) +\| \nabla
\delta n^{j+1}\|_{H^{m-1}}^2+ \| \nabla \delta c^{j+1}\|_{H^m}^2+ \|
\nabla \delta u^{j+1}\|_{H^m}^2
\]
\[
\leq C \| \delta n^{j+1}\|_{H^{m-1}}^2 +C \| \delta
c^{j+1}\|_{H^m}^2+C \| \delta u^{j+1}\|_{H^{m}}^2 +C \| \delta n^j
\|_{H^{m-1}}^2+C \|\delta c^j \|_{H^{m}}^2+C \|\delta u^j
\|_{H^{m}}^2,
\]
where $C$ depend on the $H^{m-1} \times H^m \times H^m$ norm of
($n^{j}$, $c^{j}$, $u^{j}$) and ($n^{j-1}$, $c^{j-1}$, $u^{j-1}$)
and the maximum values of $ \chi^{(i)}$ and $k^{(i)}$. Gronwall's
inequality gives us
\[
\sup_{ 0\le t \le T}( \| \delta n^{j+1} \|_{H^{m-1}}^2+\| \delta
c^{j+1} \|_{H^m}^2 +\| \delta u^{j+1} \|_{H^m}^2   )
\]
\[
\leq CT\exp \left( CT  \right)\sup_{ 0\le t \le T}( \| \delta n^{j}
\|_{H^{m-1}}^2+\| \delta c^{j} \|_{H^m}^2 +\| \delta u^{j}
\|_{H^m}^2   ).
\]
From the above inequality, we find that $( n^j,\, c^j,\, u^j)$ is a
Cauchy sequence in the Banach space $C(0, T;X_m)$ for some small
$T>0$, and thus we have the limit in the same
space.\\
\\
$\bullet$\,\,(Uniqueness)\,\, To show the uniqueness of the above
local-in-time solution, we assume that there exist two local-in-time
solutions $(c_1(x,t),\, n_1(x,t),\, u_1(x,t))$ and $(c_2(x,t),\,
n_2(x,t),\, u_2(x,t))$ of \eqref{KSNS} with the same initial data
over the time interval $[0,T]$, where $T$ is any time before the maximal
time of existence. Let $\tilde{c}(x,t):=c_1(x,t)-c_2(x,t)$,
$\tilde{n}(x,t):=n_1(x,t)-n_2(x,t)$, and
$\tilde{u}(x,t):=u_1(x,t)-u_2(x,t)$. Then  $(\tilde{c},\,
\tilde{n},\, \tilde{u})$ solves
\begin{equation}\label{KSNS-unique}
\left\{
\begin{array}{ll}
\partial_t \tilde{n} + u_1 \cdot \nabla  \tilde{n} - \Delta \tilde{n}
= -\tilde{u} \cdot \nabla n_2 -\nabla\cdot ((\chi (c_1)-\chi(c_2)) n_1 \nabla c_1)\\
\qquad\qquad -\nabla \cdot (\chi(c_2) \tilde{n} \nabla c_1)-\nabla \cdot (\chi(c_2) n_2 \nabla \tilde{c}),\\
\vspace{-2mm}\\
\partial_t \tilde{c} + u_1 \cdot \nabla \tilde{c}-\Delta \tilde{c}
=-\tilde{u} \cdot \nabla c_2-(k(c_1)-k(c_2)) n_1 -k(c_2) \tilde{n},\\
\vspace{-2mm}\\
\partial_t \tilde{u} + u_1\cdot \nabla \tilde{u} -\Delta \tilde{u} +\nabla \tilde{p}
=-\tilde{u} \cdot \nabla u_2-\tilde{n} \nabla \phi,\\
\nabla \cdot \tilde{u}=0, %\qquad t \in (0,\, T], \qquad x \in \R^d
\end{array}
\right.
 \quad\mbox{ in }\,\, (x,t)\in \R^d\times (0,\, T].
\end{equation}
Multiplying $\tilde{n}$ to both sides of the first equation of
\eqref{KSNS-unique} and integrating over $\R^d$, we have
\[
\frac12 \frac{d}{dt} \| \tilde{n} \|_{L^2}^2 + \| \nabla \tilde{n} \|_{L^2}^2 \leq \| \tilde{u} n_2 \|_{L^2} \| \nabla \tilde{n} \|_{L^2}+ \| (\chi(c_1)-\chi(c_2)) n_1 \nabla c_1 \|_{L^2} \| \nabla \tilde{n} \|_{L^2}
\]
\[
 +\| \chi(c_2) \tilde{n} \nabla c_1 \|_{L^2} \| \nabla \tilde{n} \|_{L^2}+ \| \chi(c_2) n_2 \nabla \tilde{c} \|_{L^2} \| \nabla \tilde{n}\|_{L^2}
\]
\[
\leq C \| \tilde{u} \|_{L^2}^2 \| n_2 \|_{L^{\infty}}^2+C \| \tilde{c} \|_{L^2}^2 \| n_1 \|_{L^{\infty}}^2 \| \nabla c_1 \|_{L^{\infty}}^2 +C\| \tilde{n} \|_{L^2}^2 \| \nabla c_1 \|_{L^{\infty}}^2
\]
\[
+C\| \nabla \tilde{c} \|_{L^2}^2 \| n_2 \|_{L^{\infty}}^2+ \epsilon \| \nabla \tilde{n} \|_{L^2}^2.
\]
Multiplying $\tilde{c}$ and $-\Delta \tilde{c}$ to both sides of the second equation of \eqref{KSNS-unique},  we have
\[
\frac12 \frac{d}{dt} \| \tilde{c} \|_{L^2}^2 + \| \nabla \tilde{c} \|_{L^2}^2 \leq \| \tilde{u} c_2 \|_{L^2} \| \nabla \tilde{c} \|_{L^2}
+C \| \tilde{c} \|_{L^2}^2 \| n_1 \|_{L^{\infty}}  +C \| \tilde{n} \|_{L^2} \| \tilde{c} \|_{L^2}
\]
\[
\leq C\| \tilde{u}\|_{L^2}^2 \| c_2 \|_{L^{\infty}}^2 +C \| \tilde{c} \|_{L^2}^2 \| n_1 \|_{L^{\infty}}+C\| \tilde{n}\|_{L^2}^2+C\| \tilde{c} \|_{L^2}^2 + \epsilon \| \nabla \tilde{c} \|_{L^2}^2,
\]
and
\[
\frac12 \frac{d}{dt} \| \nabla \tilde{c} \|_{L^2}^2 + \| \Delta \tilde{c} \|_{L^2}^2 \leq \| u_1 \cdot \nabla \tilde{c} \|_{L^2} \| \Delta \tilde{c} \|_{L^2} + \| \tilde{u} \cdot \nabla c_2 \|_{L^2} \| \Delta \tilde{c} \|_{L^2}
\]
\[
+\| k(c_1)-k(c_2) \|_{L^2} \| n_1 \|_{L^{\infty}} \| \Delta \tilde{c} \|_{L^2} + C\| \tilde{n} \|_{L^2} \| \Delta \tilde{c} \|_{L^2}
\]
\[
\leq C\|\nabla \tilde{c} \|_{L^2}^2 \|u_1 \|_{L^{\infty}}^2 + C \| \tilde{u}\|_{L^2}^2 \| \nabla c_2 \|_{L^{\infty}}^2 + C\| \tilde{c}\|_{L^2}^2 \| n_1 \|_{L^{\infty}}^2 +C\| \tilde{n}\|_{L^2}^2+\epsilon \| \Delta \tilde{c} \|_{L^2}^2.
\]
Multiplying $\tilde{u}$ to both sides of the third equation of
\eqref{KSNS-unique} and integrating over $\R^d$, we have
\[
\frac12 \frac{d}{dt} \| \tilde{u} \|_{L^2}^2 + \| \nabla \tilde{u} \|_{L^2}^2 \leq \| \tilde{u} \|_{L^2} \| u_2 \|_{L^{\infty}} \| \nabla \tilde{u} \|_{L^2} + C \| \tilde{n} \|_{L^2} \| \tilde{u} \|_{L^2}
\]
\[
\leq C \| \tilde{u} \|_{L^2}^2 \| u_2 \|_{L^{\infty}}^2 +C \|
\tilde{n } \|_{L^2}^2 +C \| \tilde{u}\|_{L^2}^2 +\epsilon \| \nabla
\tilde{u} \|_{L^2}^2.
\]
Summing the above estimates, we obtain
\[
\frac12 \frac{d}{dt} ( \| \tilde{n} \|_{L^2}^2 + \| \tilde{c} \|_{L^2}^2
+\| \nabla \tilde{c} \|_{L^2}^2 + \| \tilde{u} \|_{L^2}^2)\]
\[
\leq C(\| \abs{\nabla c_1}+\abs{\nabla c_2} \|_{L^{\infty}}^2+\|
n_1+n_2 \|_{L^{\infty}}^2+\|n_1\|_{L^{\infty}}^2 \| \nabla c_1
\|_{L^{\infty}}^2+\|c_2 \|_{L^{\infty}}^2+\|\abs{u_1}+\abs{u_2}
\|_{L^{\infty}}^2+1)
\]
\[
\times ( \| \tilde{n} \|_{L^2}^2 + \| \tilde{c} \|_{L^2}^2+\| \nabla
\tilde{c} \|_{L^2}^2 + \| \tilde{u} \|_{L^2}^2).\] Since all
$L^{\infty}$ norms of $( n_i, c_i, u_i)$ are controlled by  $H^{m-1}\times H^{m}\times H^{m}$ norm of $(n_i, c_i, u_i)$ with
$m \geq 3$, and the initial data of $(\tilde{c}, \tilde{n},
\tilde{u})$ are all zero, $(\tilde{c}, \tilde{n}, \tilde{u})$ are
all zero for $T>0$. That implies the uniqueness of the local
classical solution. \\

$\bullet$ (Nonnegativity) For completeness, we briefly show that $n^j$ and $c^j$ are nonnegative for all $j$. To use induction, we assume $c^j$ and $n^j$ are nonnegative. If we apply the maximum principle to the equation of $c^{j+1}$ in \eqref{linear-eq-1}, we  find that $c^{j+1}$ is nonnegative ($k(c^j)n^j$ is nonnegative).
Let us decompose $n^{j+1} =n^{j+1}_+ -n^{j+1}_-$, where
\begin{align*}
n^{j+1}_+=
\begin{cases}
n^{j+1}  \quad n^{j+1}\ge 0 \\
0\,\,\, \qquad n^{j+1} < 0,
\end{cases}\qquad
n^{j+1}_-=
\begin{cases}
-n^{j+1}  \quad n^{j+1}\le 0 \\
\quad 0\,\,\, \qquad n^{j+1} > 0.
\end{cases}
\end{align*}
Recall  that the weak derivative of   $n^{j+1}_-$ is $-\na n^{j+1}$
if $n^{j+1}_ -< 0,$ otherwise zero. It holds that
\[
\int_0^t \int_{\bbr^3} \pa_t n^{j+1}( n^{j+1})_- dxds =
 \frac 12 (\|( n^{j+1})_-(0)\|_{L^2} - \|( n^{j+1})_-(0)\|_{L^2})
 \]
since $n^{j+1}_-,  \pa_t n^{j+1} \in L^2(0,T; L^2(\bbr^2))$ (see
e.g. \cite{Ta}). Now multiplying the negative part $(n^{j+1})_{-}$
on both sides of the first equation of \eqref{linear-eq-1} and
integrating over $[0,t]\times \bbr^2$, we have
%\[
%\frac12 \frac{d}{dt} \| (n^{j+1})_{-}\|_{L^2}^2 +\| \nabla (n^{j+1})_{-}\|_{L^2}^2 \leq C \| %(n^{j+1})_{-} \|_{L^2}^2 \| \nabla c^{j} \|_{L^{\infty}}^2+\frac12 \| \nabla (n^{j+1})_{-} \|_{L^2}^2.
%\]
\[
\frac12\int_0^t  \| (n^{j+1})_{-}\|_{L^2}^2 +\| \nabla (n^{j+1})_{-}\|_{L^2}^2 ds
\leq C \int_0^t \| (n^{j+1})_{-} \|_{L^2}^2 \| \nabla c^{j} \|_{L^{\infty}}^2+\frac12 \| \nabla (n^{j+1})_{-} \|_{L^2}^ 2 ds.
\]
Using Gronwall's inequality, we have
\[
\| (n^{j+1})_{-} (t) \|_{L^2}^2 \leq \| (n^{j+1})_{-}(0) \|_{L^2}^2 \exp \left( C \int_0^t \| \nabla c^{j} \|_{L^{\infty}}^{2} ds   \right).
\]
Since the initial data $n^{j+1}_0$ is nonnegative, we conclude that $n^{j+1}$ is nonnegative.
This completes the proof.
\end{pfthm1}

%%%%%%%%%%%%%%%%%%%%%%%%%%%%%%%%%%%%%%%%%%%%%%%%%%%%%%%%%%%%%%%%%%%%%%%%%%%%%%%%%%%%%%%%%%%%%%%%%%%%%%%%%%%%%%%%%%%%%%%%%%
%%%%%%%%%%%%%%%%%%%%%%%%%%%%%%%%%%%%%%%%%%%%%%%%%%%%%%%%%%%%%%%%%%%%%%%%%%%%%%%%%%%%%%%%%%%%%%%%%%%%%%%%%%%%%%%%%%%%%%%%%%
%%%%%%%%%%%%%%%%%%%%%%%%%%%%%%%%%%%%%%%%%%%%%%%%%%%%%%%%%%%%%%%%%%%%%%%%%%%%%%%%%%%%%%%%%%%%%%%%%%%%%%%%%%%%%%%%%%%%%%%%%%
\subsection{Blow-up criterion}
Next, we observe a blow-up criterion for the fluid chemotaxis
equations.
\begin{proposition}\label{prop1}
(A Blow-up criterion) Suppose that $\chi$, $k$, $\phi$ and the initial
data $(n_0, c_0, u_0)$ satisfy all the assumptions presented in
Theorem \ref{Theorem2}. If\, $T<\infty$ is the maximal time of
existence, then
\[
\int_0^T \bke{\| \nabla u(t) \|_{L^{\infty}(\R^d)} +\| \na c(t)
\|_{L^{\infty}(\R^d)}^2} dt= \infty.
\]
\end{proposition}
\begin{proof}
At first, we consider the $L^2 $ estimate of $n$. Multiplying $n$ to
both sides of the equation of $n$ and integrating, we have
\[ \frac12 \frac{d}{dt} \| n \|_{L^2}^2 + \| \na n \|_{L^2}^2 \le C
\| \chi (c) n \na c \|_{L^2} \| \na n \|_{L^2}.
\]
Since $\chi$ is continuous and $c$ is uniformly bounded until the
maximal time of existence, we have
\[
C \| \chi (c) n \na c \|_{L^2} \| \na n \|_{L^2} \le \frac14 \| \na
n \|_{L^2}^2+ C \| \na c \|_{L^{\infty}}^2 \| n \|_{L^2}^2.
\]
For the estimates of $c$, we use the calculus inequality
\[
\| \na( u \cdot \na c )-(u\cdot \na)\na c\|_{L^2} \le C \| \na u
\|_{L^{\infty}} \| \na c \|_{L^2}.
\]
Multiplying $-\Delta c$ to both sides of the equation of $c$ and
integrating, we obtain
\[
\frac12 \frac{d}{dt} \| \nabla  c\|_{L^2}^2 + \|  \Delta c
\|_{L^2}^2  \le C \|\na u \|_{L^{\infty}} \| \na c \|_{L^2}^2 + C\|
( k(c) n ) \|_{L^2}^2 +\frac14 \|  \Delta c \|_{L^2}^2.
\]
%We estimate the second term in the above right hand side by using the Young's inequality.
%\[
%C \| \na c \|_{L^{\infty}} \| \na^2 u \|_{L^2} \| \na^2 c \|_{L^2}
%\le \epsilon \| \na^2 u \|_{L^{2}}^2 + C \| \na c \|_{L^\infty}^2 \|
%\na^2 c \|_{L^2}^2.
%\]
%Since we assume $k\in C^{m-1}$,  we have
%\[
%\| \na ( k(c) n )\|_{L^2}^2 \le C \| \na c \|_{L^{\infty}}^2 \| n
%\|_{L^2}^2 + \frac{C_1}{4} \| \na n \|_{L^2}^2,
%\]
%for some constant $C_1$. Thus we have
%\[
%\frac12 \frac{d}{dt} \| \Delta  c\|_{L^2}^2 + \| \na \Delta c
%\|_{L^2}^2 \le  C \|\na u \|_{L^{\infty}} \| \na^2 c \|_{L^2}^2
%\]
%\[
%+\epsilon \| \na^2 u \|_{L^{2}}^2 + C \| \na c \|_{L^\infty}^2 \|
%\na^2 c \|_{L^2}^2+C \| \na c \|_{L^{\infty}}^2 \| n \|_{L^2}^2 +
%C_1 \| \na n \|_{L^2}^2+\frac14 \| \na \Delta c \|_{L^2}^2.
%\]
For the equations of $u$, multiplying $-\Delta u$ to both sides
of the equations and integrating by parts, we have
\[
\frac12 \frac{d}{dt} \| \na u \|_{L^2}^2 + \|\Delta u \|_{L^2}^2 \le
C \norm{ \na u }_{L^{\infty}}\norm{\na u }_{L^2}^2+C\norm{n}_{L^2}
\norm{\Delta u}_{L^2}.
\]
Collecting all the estimates, we obtain
\[
 \frac{d}{dt}\left( \| n\|_{L^2}^2 + \| \nabla c
\|_{L^2}^2 + \| \na u \|_{L^2}^2  \right)+ \left( \| \na n
\|_{L^2}^2 + \| \Delta c \|_{L^2}^2 + \|\Delta u \|_{L^2}^2 \right)
\]
\[
\le C \left( \|\nabla c\|_{L^{\infty}}^2 +\| \nabla u
\|_{L^{\infty}}
 \right) \left(\|n\|_{L^2}^2+ \| \na c \|_{L^2}^2+\|\na u\|_{L^2}^2\right).
\]
From Gronwall's inequality, we have
\[
 \sup\left(  \| n\|_{L^2}^2 + \| \na c \|_{L^2}^2 + \|
\na u \|_{L^2}^2  \right)+ \int_0^T\left( \| \na n \|_{L^2}^2 + \|
\Delta c \|_{L^2}^2 +\|\Delta u \|_{L^2}^2 \right)dt
\]
\[
\le C( \| n_0\|_{L^2}^2+ \| \na c_0 \|_{L^2}^2 + \| \na u_0
\|_{L^2}^2) \exp \left ( \int_0^T \| \na u \|_{L^{\infty}} + \| \na
c \|_{L^{\infty}}^2 dt \right).
\]
Note that $\| n \|_{L^{\infty}(0, T; L^2)}$ and $\| \nabla n
\|_{L^{2}(0, T; L^2)}$ are uniformly bounded if $\int_0^T \| \nabla
u \|_{L^{\infty}}+ \| \nabla c \|_{L^{\infty}}^2 dt$ is bounded.
Moreover, we see that $n\in L^q_xL^{\infty}_t$ and $\nabla
n^{q/2}\in L^2_xL^2_t$ for all $2<q<\infty$. Indeed,
\[
\frac{d}{dt}\norm{n}^q_{L^q}+\norm{\nabla
n^{\frac{q}{2}}}_{L^2}^2\leq C\int_{\R^2}\abs{n\nabla c\nabla
n^{q-1}}dx\leq C\norm{\nabla
c}^2_{L^{\infty}}\norm{n}^q_{L^q}+\frac{1}{2}\norm{\nabla
n^{\frac{q}{2}}}^2_{L^2}.
\]
From the above inequality, we have $\| n(t) \|_{L^q} \leq C$, where
$C$ is independent of $q$.
Letting $q \rightarrow \infty$,  we have $n \in L^{\infty}_x L^{\infty}_t$.\\
Next, we consider the estimate in the space $(n,c) \in H^1 \times H^2
$. We have
\[
\frac12 \frac{d}{dt} \norm{ \na n }_{L^2}^2 + \norm{\Delta n
}_{L^2}^2 \leq C \norm{\na u }_{L^{\infty}} \norm{\na n }_{L^2}^2 +C
\norm{\na n }_{L^2} \norm{\na c}_{L^{\infty}} \norm{\na^2 n }_{L^2}
\]
\[
+C \norm{ n}_{L^{\infty}} \norm{\Delta c }_{L^2} \norm{ \na^2 n
}_{L^2} +C \norm{n}_{L^{\infty}} \norm{ \na c
}_{L^{\infty}}\norm{\na c}_{L^2} \norm{\na^2 n }_{L^2}.
\]
From Young's inequality and Gronwall's inequality, we have
\[
\sup \norm{\na n}_{L^2}^2 +\int_0^T \norm{\na^2 n }_{L^2}^2 dt\]
\[
\leq \left(\norm{\na n_0 }_{L^2}^2 +C\norm{n}_{L^{\infty}(0,T;
L^{\infty})}\left(\int_0^T \norm{\Delta c}_{L^2}^2 dt +\norm{\na c
}_{L^{\infty}(0,T; L^{2})} \int_0^T \norm{\na c}_{L^{\infty}}^2 dt
\right)\right)
\]
\[
\times \exp \left ( \int_0^T \| \na u \|_{L^{\infty}} + \| \na c
\|_{L^{\infty}}^2 dt \right).
\]
Hence, $n \in H^1_x L^{\infty}_t \cap H^2_x L^2_t$.  For the $H^2$
estimate of $c$, we have
\[
\frac12\frac{d}{dt} \norm{\Delta c}_{L^2}^2 + \norm{\na \Delta
c}_{L^2}^2 \leq C\norm{\na u }_{L^{\infty}}\norm{\Delta c}_{L^2}^2 +
C\norm{\Delta u}_{L^2} \norm{c}_{L^{\infty}} \norm{\na \Delta
c}_{L^2}
\]
\[
+C\norm{\na c}_{L^2} \norm{n}_{L^{\infty}} \norm{\na \Delta c}_{L^2}
+C \norm{\na n}_{L^2} \norm{\na \Delta c}_{L^2}.
\]
By Gronwall's inequality, we have $c \in H^2_xL^{\infty}_t \cap
H^3_xL^2_t$. Similarly,  $u\in H^2_xL^{\infty}_t \cap
H^3_xL^2_t$. Then, we consider the estimate in the space $(n, c, u) \in H^2 \times H^3\times H^3$. Proceeding similarly to the above,  we obtain
\[
\frac12\frac{d}{dt}\|  n \|_{H^{2}}^2 + \| \na n \|_{H^{2}}^2
\leq C\|  u \|_{L^{\infty}} \| n \|_{H^{2}}\| \na n
\|_{H^{2}}+C\| \na u\|_{L^{\infty}} \|  n \|_{H^{1}} \| \na n
\|_{H^{2}}
\]
\[
+\frac14 \| \na n \|_{H^{2}}^2+ C \| \chi (c) n \na c
\|_{H^{2}}^2.
\]
In the above, the last term can be controlled by
\[
\| \chi (c) n \na c
\|_{H^{2}} \leq C \norm{n}_{H^2} \norm{ \chi(c) \na c}_{H^2},
\]
and
\[
\norm{ \na^2( \chi (c) \na c) }_{L^2} \leq C\norm { \na^3 c}_{L^2} + C\norm{ \na^2 c}_{L^2} \norm{\na c}_{L^{\infty}} + C\norm{ \na c}_{L^6}^3.
\]
We already obtained $c \in H^2_xL^{\infty}_t \cap
H^3_xL^2_t$. Hence, if we use Young's inequality and Gronwall's inequality, we have
\[
\sup \norm{ n}_{H^2}^2 +\int_0^T \norm{\na n }_{H^2}^2 dt
%\]
%\[
\leq \norm{ n_0 }_{H^2}^2
 \exp \left (C+ \int_0^T \| \na u \|_{L^{\infty}} + \| \na c
\|_{L^{\infty}}^2 dt \right).
\]
Similarly, we estimate $c$ as
\[
\frac12 \frac{d}{dt} \|  c\|_{H^3}^2 + \| \na c \|_{H^3}^2 \le C
\|\na u \|_{L^{\infty}} \|  c \|_{H^3}\| \na c \|_{H^3} + C  \|  u
\|_{H^3} \| \na c \|_{L^\infty}\norm{\na c}_{H^3}
\]
\[
+ C\| ( k(c) n ) \|_{H^{2}}^2 +\frac14 \|  \na c \|_{H^3}^2.
\]
%From the induction hypothesis, we have $\| \nabla^{m-1}
%k(c)\|_{L^{\infty}} <C(T)$, i.e.,
We can
control the term $\| ( k(c) n ) \|_{H^{2}}^2$ by $C(\|
c\|_{H^2}^{2}\norm{n}_{H^2}^2 + \norm{c}_{H^1}^2 \norm{\na c}_{L^{\infty}}^2 \norm{n}_{H^2}^2)$. For the estimate of $u$, we have
\[
\frac12 \frac{d}{dt} \|  u \|_{H^3}^2 + \|\na u \|_{H^3}^2 \le C\|
\na u \|_{L^{\infty}} \|  u \|_{H^3}\norm{ \na u}_{H^{3}} +
\frac14\| \na u \|_{H^3}^2 + C \| n \|_{H^{2}}^2.
\]
Using Gronwall's inequality, we have $(c,u) \in (H^3_{x}L^{\infty}_{t} \cap H^{4}_{x} L^{2}_{t} ) \times (H^3_{x}L^{\infty}_{t} \cap H^{4}_{x} L^{2}_{t} )$.
%Assume the following hold up tp $m\ge 4$,
%\[ \frac 12\ddt (\norm{n}^2_{H^{m-2}} + \norm{c}^2_{H^{m-1}} + \norm{u}^2_{H^{m-1}})
%+\int_0^T \norm{n}^2_{H^{m-1}} + \norm{c}^2_{H^{m}} +
%\norm{u}^2_{H^{m}} dt \]
%\[\le C(T, \norm{n_0}^2_{H^{m-2}}, \norm{c_0}^2_{H^{m-1}},\norm{u_0}^2_{H^{m-1}}).\]
% and $c \in H^3_xL^{\infty}_t \cap H^4_xL^2_t$. We omit the details. \\
 Let us  consider $H^{m-1}\times H^{m} \times H^{m}$ estimates. The case $m=2,3$ and 4 are proved in the above, hence we consider the $m \geq 5$ case.
 Taking $\pa^{\alpha}$ ($|\alpha|\leq m-1$) and multiplying $\pa^{\alpha}
n$ to both sides of the equation $n$ and integrating and
summing, we have
\[
\frac12\frac{d}{dt}\|  n \|_{H^{m-1}}^2 + \| \na n \|_{H^{m-1}}^2
\leq C\| \na u \|_{L^{\infty}} \| n \|_{H^{m-1}}\| \na n
\|_{H^{m-1}}+C\|  u\|_{H^{m-1}} \| \na n \|_{L^{\infty}} \| \na n
\|_{H^{m-1}}
\]
\[
+\frac14 \| \na n \|_{H^{m-1}}^2+ C \| \chi (c) n \na c
\|_{H^{m-1}}^2.
\]
We already obtained the estimate for the case $m=4$, thus $\|\na c\|_{L^{\infty}(0,T; L^{\infty})}$ is bounded. Hence, we have
\[ \|\na \chi(c)\|_{H^{m-2}} \le C(1+ \|\na c\|_{L^{\infty}}) \|\na \chi'(c)\|_{H^{m-3}}.\]
Using the classical product lemma on each step of iteration, we can
control
\[ \|\na\chi(c)\|_{H^{m-2}} \le C(1+ \|\na c\|_{L^{\infty}})^{m-1}.\]
Then  we have
\[\norm{\chi(c) n \na c}_{H^{m-1}} \leq C (1+\| c\|_{H^{m}} + \|\na c\|_{L^{\infty}}^{m}) \norm{n}_{H^{m-1}}\] using the product lemma.
% Using calculus inequality again, we obtain
% \[
% \norm{ \chi(c) \na c}_{H^{m-1}}\leq C\norm{\na c}_{H^{m-1}} +C \norm{\na c}_{L^{\infty}} \norm{ \na \chi (c)}_{H^{m-2}}.
% \]
% Once again using calculus inequality, we have
% \[
% \norm{\na \chi(c) }_{H^{m-2}}\leq C \norm{\na c}_{H^{m-2}}+C \norm{\na c}_{L^{\infty}} \norm{\na \chi'(c)}_{H^{m-3}}.
% \]
% If we use calculus inequality repeatedly, we can control ....\\
For the $H^{m}$ estimate of $c$, we proceed similarly to have
\[
\frac12 \frac{d}{dt} \|  c\|_{H^m}^2 + \| \na c \|_{H^m}^2 \le C
\|\na u \|_{L^{\infty}} \|  c \|_{H^m}\| \na c \|_{H^m} + C  \|  u
\|_{H^m} \| \na c \|_{L^\infty}\norm{\na c}_{H^m}
\]
\[
+ C\| ( k(c) n ) \|_{H^{m-1}}^2 +\frac14 \|  \na c \|_{H^m}^2.
\]
%From the induction hypothesis, we have $\| \nabla^{m-1}
%k(c)\|_{L^{\infty}} <C(T)$, i.e.,
As is shown for the term $\norm{\chi(c) n \na c}_{H^{m-1}}$, we
control the term $\| ( k(c) n ) \|_{H^{m-1}}^2$ by $C(1+\|\na
c\|_{L^{\infty}}^{2m-2})\norm{n}_{H^{m-1}}^2$. For the estimate of $u$, we have
\[
\frac12 \frac{d}{dt} \|  u \|_{H^m}^2 + \|\na u \|_{H^m}^2 \le C\|
\na u \|_{L^{\infty}} \|  u \|_{H^m}\norm{ \na u}_{H^{m}} +
\frac14\| \na u \|_{H^m}^2 + C \| n \|_{H^{m-1}}^2.
\]
Thus, by collecting all the above estimates and using Gronwall's
inequality, we have $(n, c, u) \in (H^{m-1}_x L^{\infty}_t \cap
H^{m}_xL^2_t)\times (H^{m}_xL^{\infty}_t \cap H^{m+1}_x L^2_t)\times
(H^{m}_xL^{\infty}_t \cap H^{m+1}_x L^2_t) $. This completes the
proof.
\end{proof}

%Theorem $2$ states that the velocity part in the blow-up criterion
%can be refined in both two and three dimensions, as is consistent
%with  the blow-up criterion for the incompressible Navier-Stokes
%equations.

%Especially by using the Beale-Kato-Majda's logarithmic Sobolev type inequality\cite{BKM}, we obtain the blowup criterion in terms of the vorticity. For the two dimensional equations and the chemotaxis equations coupled with the Navier-Stokes equations, we can have the more refined blowup criterion.

%If we use the Beale-Kato-Majda's logarithmic inequality, then we
%have the blowup criterion in terms of
%\[
%\int_0^T \| \omega \|_{L^{\infty}} +\| \na c \|_{L^{\infty}}^2 dt
%\]
%{\it ***********************$L^2 \times H^2 \times H^1$
%estimates**********}\\

We are ready to present the proof of Theorem \ref{Theorem3}.

\begin{pfthm2}
In the proof of Proposition $1$, we notice that
 $\|\na c\|_{L^{\infty}}$ is solely responsible for $n\in L^2_xL^{\infty}_t$ and $\nabla n\in
L^2_xL^2_t$.  Indeed,
\begin{align}\label{nc}
\frac{d}{dt}\norm{n}^2_{L^2}+\norm{\nabla n}_{L^2}^2\leq
C\int_{\R^d}\abs{n\nabla c\nabla n}dx\leq C\norm{\nabla
c}^2_{L^{\infty}}\norm{n}^2_{L^2}+\frac{1}{2}\norm{\nabla
n}^2_{L^2}.
\end{align}
This implies $u\in L^2_xL^{\infty}_t$ and $\nabla u\in L^2_xL^2_t$ by
\begin{align}\label{uc}
\frac{d}{dt}\norm{u}^2_{L^2}+\norm{\nabla u}_{L^2}^2\leq
C\|n\|_{L^2}\|u\|_{L^2}.
\end{align}
Moreover, we have $n\in L^q_xL^{\infty}_t$ and $\nabla n^{q/2}\in
L^2_xL^2_t$ for all $2<q<\infty$;
\[
\frac{d}{dt}\norm{n}^q_{L^q}+\norm{\nabla
n^{\frac{q}{2}}}_{L^2}^2\leq C_q\int_{\R^d}\abs{n\nabla c\nabla
n^{q-1}}dx\leq C_q\norm{\nabla
c}^2_{L^{\infty}}\norm{n}^q_{L^q}+\frac{1}{2}\norm{\nabla
n^{\frac{q}{2}}}^2_{L^2}.
\]
Next, we see that $\nabla c\in L^2_xL^{\infty}_t$ and $\nabla^2
c\in L^2_xL^2_t$. Indeed,
\begin{align}\label{cc}
\frac{d}{dt}\norm{\nabla c}^2_{L^2}+\norm{\nabla^2 c}_{L^2}^2\leq
C\norm{\nabla c}_{L^{\infty}}\norm{u}_{L^2}\norm{\nabla^2
c}_{L^2}+C\norm{n}_{L^2}\norm{\nabla^2 c}_{L^2}.
\end{align}
We first consider the two-dimensional case.\\
\\
$\bullet$ \,\,({\bf{2D case}})\qquad For convenience, we denote
vorticity as $\omega:=\nabla\times u$; that is $\omega =
\partial_1 u_2- \partial_2 u_1$ in two dimensions. Next, we consider the vorticity equation
\[
\omega_t-\Delta\omega+u\nabla \omega=-\nabla^{\perp}n\nabla\phi,
\]
where $\nabla^{\perp}n=(-\partial_2 n, \partial_1 n)$. We note that
$\omega\in L^2_xL^{\infty}_t$ and $\nabla\omega\in L^2_xL^2_t$,
since
\begin{equation}\label{est-omega-10}
\frac{d}{dt}\norm{\omega}^2_{L^2}+\norm{\nabla \omega}_{L^2}^2\leq
C\norm{\nabla n}_{L^2}\norm{\omega}_{L^2}.
\end{equation}
Furthermore, we observe that $\nabla\omega\in L^2_xL^{\infty}_t$
and $\nabla^2\omega\in L^2_xL^2_t$. Indeed, testing $-\Delta
\omega$, we get
\[
\frac{d}{dt}\norm{\nabla\omega}^2_{L^2}+\norm{\nabla^2
\omega}_{L^2}^2\leq
\norm{u}_{L^4}\norm{\nabla\omega}_{L^4}\norm{\Delta\omega}_{L^2}+\norm{\nabla
n}_{L^2}\norm{\Delta\omega}_{L^2}
\]
\[
\leq C\norm{u}^{\frac{1}{2}}_{L^2}\norm{\nabla
u}^{\frac{1}{2}}_{L^2}\norm{\nabla\omega}^{\frac{1}{2}}_{L^2}
\norm{\nabla^2\omega}^{\frac{3}{2}}_{L^2}+\norm{\nabla
n}_{L^2}\norm{\Delta\omega}_{L^2}.
\]
Therefore, via embedding, we have
\[
\int_0^T\norm{\nabla u}_{L^{\infty}}dt\leq \int_0^T\norm{\nabla
u}_{H^{2}}dt\leq C\int_0^T\norm{\omega}_{H^{2}}dt<\infty.
\]
This completes the proof of the 2D case.\\
\\
$\bullet$\,\,({\bf{3D case}})\quad We will show this case by
contradictory arguments. We suppose that the condition
\eqref{3d-nse-reg} is not true. We first recall the vorticity equation
\[
\omega_t-\Delta\omega+u\nabla \omega=\omega\nabla
u-\nabla^{\perp}n\nabla\phi.
\]
Under the condition \eqref{3d-nse-reg} we have $\omega\in
L^2_xL^{\infty}_t$ and $\nabla\omega\in L^2_xL^2_t$ as follows. We
denote $Q^*=\R^3\times (T^*-\delta,t)$ for $T^* -\delta < t< T^*$.
For any given $p,q$ satisfying $ 3/p+ 2/q=1$, $3< p\le \infty$, we
choose $l,m$ such that $1/p+1/l=1/2$ and $1/q+1/m=1/2$. We then
remind that, due to the Gargliardo-Nirenberg's inequality,
\[
\norm{u}_{L^{l,m}_{x,t}}\leq
C\norm{u}^{\theta}_{L^{2,\infty}_{x,t}}\norm{\nabla
u}^{1-\theta}_{L^{2,2}_{x,t}},\qquad 2\leq l\leq
6,\,\,\,\frac{3}{l}+\frac{2}{m}=\frac{3}{2},
\]
where $\theta=(6-l)/2l$ and $1-\theta=(3l-6)/2l$. Then we have
\[
\frac{d}{dt}\norm{\omega}^2_{L^2}+\norm{\nabla \omega}_{L^2}^2\leq
\norm{u}_{L^p}\norm{\omega}_{L^l}\norm{\nabla\omega}_{L^2}
+C\norm{n}_{L^2}\norm{\nabla\omega}_{L^2}.
\]
Next, integrating in time over $(T^*-\delta, t)$,
\begin{align*}
\norm{\omega(t)}^2_{L^2}+ \int_{T^*-\delta}^t\norm{\nabla
\omega}_{L^2}^2 & \leq \norm{\omega(T^*-\delta)}^2_{L^2}+
C\norm{u}_{L^{p,q}_{x,t}}\norm{\omega}_{L^{l,m}_{x,t}}\norm{\nabla\omega}_{L^{2,2}_{x,t}}
+C\norm{n}_{L^{2,2}_{x,t}}\norm{\nabla\omega}_{L^{2,2}_{x,t}}
\\&
\leq \norm{\omega(T^*-\delta)}^2_{L^2}+
\norm{u}_{L^{p,q}_{x,t}}\norm{\omega}^{\theta}_{L^{2,\infty}_{x,t}}\norm{\nabla
\omega}^{2-\theta}_{L^{2,2}_{x,t}}+C\norm{n}_{L^{2,2}_{x,t}}\norm{\nabla\omega}_{L^{2,2}_{x,t}}.
\end{align*}
Note that $\theta>0$. By Young's inequality, we have
\begin{align*}
\norm{\omega(t)}^2_{L^2}+& \int_{T^*-\delta}^t\norm{\nabla
\omega}_{L^2}^2 \leq \norm{\omega(T^*-\delta)}^2_{L^2}+
C\norm{u}^{\frac{2}{\theta}}_{L^{p,q}_{x,t}}\norm{\omega}^{2}_{L^{2,\infty}_{x,t}}
+C\norm{n}^2_{L^{2,2}_{x,t}}+\frac{1}{2}\norm{\nabla\omega}^2_{L^{2,2}_{x,t}}.
\end{align*}
Since $\norm{u}_{L^{p,q}_{x,t}}$ can be sufficiently small in
$(T^*-\delta, T^*) \times \bbr^3$ by decreasing $\delta$, we have
\begin{align}\label{del}
\norm{\omega(t)}^2_{L^2}+\frac{1}{2}\int_{T^*-\delta}^t\norm{\nabla
\omega}_{L^2}^2\leq
\norm{\omega(T^*-\delta)}^2_{L^2}+C\norm{n}^2_{L^{2,2}_{x,t}},
\end{align}
which is bounded by \eqref{nc}. Since $t$ is arbitrary for all
$t<T^*$, this estimate is uniform.

Next, we observe that $\nabla^2 c\in L^2_xL^{\infty}_t$ and $\nabla^3
c\in L^2_xL^2_t$. Indeed, we estimate
\[
\frac{d}{dt}\norm{\nabla^2 c}^2_{L^2}+\norm{\nabla^3 c}_{L^2}^2\leq
C(\norm{\nabla n}_{L^2}+\norm{\nabla c}_{L^{\infty}}\norm{\nabla
u}_{L^2})\norm{\nabla^3 c}_{L^2}+\norm{u}_{L^6}\norm{\nabla^2
c}_{L^3}\norm{\nabla^3 c}_{L^2}
\]
\[
\leq C(\norm{\nabla n}_{L^2}+\norm{\nabla
c}_{L^{\infty}}\norm{\nabla u}_{L^2})\norm{\nabla^3
c}_{L^2}+\norm{\omega}_{L^2}\norm{\nabla^2
c}^{\frac{1}{2}}_{L^2}\norm{\nabla^3 c}^{\frac{3}{2}}_{L^2},
\]
and use \eqref{nc}, \eqref{uc}. Similarly, we show that $n\in
L^{\infty}_tH^1_x\cap L^2_tH^2_x$ by estimating
\[
\frac{d}{dt}\norm{\nabla n}^2_{L^2}+\norm{\nabla^2 n}_{L^2}^2\leq
\norm{u}_{L^6}\norm{\nabla n}_{L^3}\norm{\nabla
n}_{L^2}+\norm{\nabla c}_{L^{\infty}}\norm{\nabla
n}_{L^2}\norm{\Delta n}_{L^2}
\]
\[
+\norm{u}_{L^6}\norm{\nabla^2 c}_{L^3}\norm{\nabla^2
n}_{L^2}+\norm{\nabla c}_{L^{\infty}}\norm{n}_{L^6}\norm{\nabla
c}_{L^3}\norm{\Delta n}_{L^2}.
\]
Finally, we show that $\omega\in H^1_xL^{\infty}_t\cap H^2_xL^2_t$.
Testing $-\Delta \omega$ to the equations, we have
\[
\frac{d}{dt}\norm{\nabla\omega}^2_{L^2}+\norm{\nabla^2
\omega}_{L^2}^2\leq
\norm{u}_{L^p}\norm{\nabla\omega}_{L^l}\norm{\nabla^2\omega}_{L^2}
\]
\[
+\norm{\nabla
u}_{L^4}\norm{\omega}_{L^4}\norm{\nabla^2\omega}_{L^2}+C\norm{\nabla
n}_{L^2}\norm{\nabla^2\omega}_{L^2},
\]
where $3<p\le \infty$ and $1/p + 1/l = 1/2$. Note that, via the
Gargliardo-Nirenberg's inequality,
\[
\norm{\nabla
u}_{L^4}\norm{\omega}_{L^4}\norm{\nabla^2\omega}_{L^2}\leq
C\norm{\omega}^2_{L^4}\norm{\nabla^2\omega}_{L^2}\leq
C\norm{\omega}^{\frac{5}{4}}_{L^2}\norm{\nabla^2\omega}^{\frac{7}{4}}_{L^2}.
\]
We treat the term
$\norm{u}_{L^p}\norm{\nabla\omega}_{L^l}\norm{\nabla^2\omega}_{L^2}$
similarly to
$\norm{u}_{L^p}\norm{\omega}_{L^l}\norm{\nabla\omega}_{L^2}$ in the
estimation of \eqref{del}. Therefore, since $\nabla^2\omega\in
L^2_xL^2_t$, we have
\[
\int_0^T\norm{\nabla u}_{L^{\infty}}dt\leq \int_0^T\norm{\nabla
u}_{H^{2}}dt\leq C\int_0^T\norm{\omega}_{H^{2}}dt<\infty.
\]
This completes the proof.
\end{pfthm2}

%%%%%%%%%%%%%%%%%%%%%%%%%%%%%%%%%%%%%%%%%%%%%%%%%%%%%%%%%%%%%%%%%%%%%%%%%%%%%%%
%%%%%%%%%%%%%%%%%%%%%%%%%%%%%%%%%%%%%%%%%%%%%%%%%%%%%%%%%%%%%%%%%%%%%%%%%%%%%%%
%%%%%%%%%%%%%%%%%%%%%%%%%%%%%%%%%%%%%%%%%%%%%%%%%%%%%%%%%%%%%%%%%%%%%%%%%%%%%%%
\section{Global solutions in two dimensions}\label{section4}
In this section, we provide the proof of global existence of smooth
solutions in time with large initial data in two dimensions. For the
proof of Theorem \ref{Theorem4}, we show some a priori estimates,
which are uniform until the maximal time of existence. Moreover,
such estimates imply that the blow-up condition quantity in Theorem
\ref{Theorem3} is uniformly bounded up to the maximal time of
existence. Therefore, the maximal time cannot be finite. Now we
present the proof of Theorem~\ref{Theorem4}.

\begin{pfthm3} We first present the following estimates for the solutions
to the two-dimensional chemotaxis system coupled with the
Navier-Stokes equations.
\begin{equation}\label{n-eq-est}
n(1+|x| +|\ln n|) \in L^{\infty}(0, \, T; L^{1} (\R^2)),\quad \nabla
\sqrt{n} \in L^{2}(0, \, T; L^{2} (\R^2)),
\end{equation}
\begin{equation}\label{c-eq-est}
c \in L^{\infty} (0,\, T; L^{1}(\R^2)\cap L^{\infty}(\R^2) \cap
H^{1} (\R^2)),\quad \nabla c \in L^{2} (0, \, T; L^{2}(\R^2)),
\end{equation}
%\[
%\sqrt{n} |\nabla c| \in L^{2}(0,\,T; L^{2}(\R^{2})),
%\]
\begin{equation}\label{u-eq-est}
u \in L^{\infty}(0,\, T; L^{2}(\R^2)) ,\quad \nabla u \in L^{2}(0,\,
T; L^{2}(\R^{2})).
\end{equation}
We have the mass conservation for $n(t,x)$ as\\
\begin{equation}\label{mass-conservation-n}
\int_{\R^{2}} n(t,\,x) dx=\int_{\R^2} n_0(x) dx.
\end{equation}
Multiplying $c^{q-1}(t,x)$ to both sides of the second equation
of \eqref{KSNS} and integrating over $\R^2$, we have
\begin{equation}\label{energy-estimate-c-1}
\frac{1}{q} \frac{d}{dt} \| c\|_{L^q}^q +\frac{4(q-1)}{q^2}\| \nabla
c^{\frac{q}{2}} \|_{L^2}^2 +\int_{\R^2} k(c)n c^{q-1} dx=0.
\end{equation}
Hence, we have $c \in L^{\infty}(0,\, T; L^q)$ for any $1<q \le
\infty$ and $ \nabla c^{\frac{q}{2}} \in L^2(0,\, T; L^2)$ for any
$1<q < \infty$. Multiplying $\ln n$ to both sides of the first
equation of \eqref{KSNS} and integrating over $\R^{2}$, we have
\begin{equation}\label{estimates-nlnn}
\frac{d}{dt}\int_{\R^2} n \ln n dx + 4 \int_{\R^2} |\nabla \sqrt{n}
|^2 dx+\int_{\R^{2}} \chi'(c) |\nabla c|^2 n dx = -\int_{\R^2} \chi
(c) \Delta c n dx.
\end{equation}
Multiplying $-\Delta c$ to both sides of \eqref{KSNS} and
integrating over $\R^2$, we obtain
\[
\frac{d}{dt} \| \nabla c \|_{L^2}^2+ \| \Delta c \|_{L^2}^2 =
\int_{\R^2} k(c) \Delta c \,\, n dx+\sum_{j,k}\int_{\R^2}
c\partial_k u_i \partial_i \partial_k c dx
\]
\bq \label{H1-estimate-c} \le  \int_{\R^2} k(c) \Delta c \,\, n dx+
C_1 \| \nabla u \|_{L^2} \| c\|_{L^{\infty}}\| \Delta c\|_{L^2}.
  \eq
Multiplying $\mu$ to both sides of \eqref{H1-estimate-c} and
then adding \eqref{estimates-nlnn}, we have
\[
\frac{d}{dt} \int_{\R^2} n \ln n + \mu |\nabla c |^2 dx+
\int_{\R^2}4 | \nabla \sqrt{n} |^2 +\mu  |\Delta c|^2dx
\]
\[ \le \epsilon \| \Delta c\|_{L^2} \| \sqrt{n} \|_{L^4}^2 +C_2
\|c\|_{L^{\infty}}^2 \| \nabla u \|_{L^2}^2 +\frac14\mu \| \Delta
c\|_{L^2}^2
\]
\begin{equation} \label{estimate-nlnn-c}
\le \frac12\mu\| \Delta c\|_{L^2}^2+\epsilon C_3 \| \nabla
\sqrt{n}\|_{L^2}^2  +C_2 \|c\|_{L^{\infty}}^2 \| \nabla u
\|_{L^2}^2,
\end{equation}
where we used the condition (A). Here we choose $\epsilon$ to be so
small that $\epsilon C_3 <2$ and, for convenience, set
$\frac{\lambda_1}{2}:=C_2 \| c_0 \|_{L^{\infty}}^2$. On the other
hand, multiplying $u$ to both sides of the third equations of
\eqref{KSNS} and integrating over $\R^2$, we have
\begin{equation}\label{L2-estimate-of-u}
\frac12\frac{d}{dt} \| u\|_{L^2}^2 +\|
\nabla u \|_{L^2}^2 =-\int_{\R^2} n \nabla \phi u dx.
\end{equation}
Multiplying $\phi$ to both sides of the first equation of
\eqref{KSNS} and integrating over $\R^2$, we have
\[
\frac{d}{dt} \int_{\R^2} n \phi dx= - \int_{\R^2} u \cdot \nabla n
\phi dx - \int_{\R^2} \nabla n \cdot \nabla \phi dx + \int_{\R^2}
\chi (c) n \nabla c \cdot \nabla \phi dx
\]
\begin{equation}\label{estimate-of-nphi}
\le -\int_{\R^2} u \cdot \nabla n \phi dx+C_4 \| \nabla
\sqrt{n}\|_{L^2} \|\sqrt{n}\|_{L^2}+C_5 \| n \|_{L^2} \| \nabla
c\|_{L^2}.
\end{equation}
Summing \eqref{L2-estimate-of-u} and \eqref{estimate-of-nphi}, we have
\[
\frac{d}{dt}\left(\frac12 \| u\|_{L^2}^2+\int_{\R^2} n \phi
dx\right)+ \|\nabla u \|_{L^2}^2
\]
\begin{equation}\label{estimate-u-nphi}
\le C_5\| \sqrt{n}\|_{L^2} \|\nabla \sqrt{n}\|_{L^2}+ C_6 \|
\sqrt{n} \|_{L^2} \| \nabla \sqrt{n} \|_{L^2} \| \nabla c \|_{L^2}.
\end{equation}
Multiplying $\lambda_1$ to both sides of \eqref{estimate-u-nphi}
and adding \eqref{estimate-nlnn-c}, we obtain
\[
\frac{d}{dt}\int_{\R^2} n \ln n + \mu |\nabla c
|^2+\frac{\lambda_1}{2} |u|^2+\lambda_1 n \phi dx+\int_{\R^2} 2
|\nabla \sqrt{n}|^2+\frac{\mu}{2} |\Delta c|^2 +\frac{\lambda_1}{2}
|\nabla u |^2 dx
\]
\[
\le \lambda_1 C_5 \| n_0 \|_{L^1}^\frac12 \| \nabla \sqrt{n}
\|_{L^2} + \lambda_1 C_6 \| n_0 \|_{L^1}^\frac12 \| \nabla \sqrt{n}
\|_{L^2} \| \nabla c\|_{L^2}
\]
\begin{equation}\label{estimate-weak-solution}
\le \frac{\lambda_1^2 C_5^2}{2} \| n_0 \|_{L^1} +\frac{\lambda_1^2
C_6^2}{2} \| n_0 \|_{L^1} \| \nabla c\|_{L^2}^2+\| \nabla \sqrt{n}
\|_{L^2}^2.
\end{equation}
Using  Gronwall's inequality, we have
\[
\sup_{ 0 \leq t \leq T} \left(\int_{\R^2} n \ln n + \mu |\nabla c
|^2+\frac{\lambda_1}{2} |u|^2+\lambda_1 n \phi dx\right)
\]
\[
+\int_0^T\int_{\R^2} |\nabla \sqrt{n}|^2+\frac{\mu}{2} |\Delta c|^2
+\frac{\lambda_1}{2} |\nabla u |^2 dxdt    \leq C(T).
\]
Next, we show that $n \abs{\ln n}\in L^{\infty}(0,\, T; L^2(\R^2))$,
following a typical argument for dealing with kinetic entropy (see
e.g. \cite{DiLi}).
 We first note that
\begin{equation}\label{log-estimate}
\int_{\R^2} n(\ln n)_-\leq C+C\int_{\R^2} n\wx,
\end{equation}
where $(\ln n)_-$ is a negative part of $\ln x$ and
$\wx=(1+\abs{x}^2)^{\frac{1}{2}}$. Indeed, setting $D_1=\{x:
n(x)\leq e^{-\abs{x}}\}$ and $D_2=\{x: e^{-\abs{x}}<n(x)\leq 1\}$,
we have
\begin{align}\label{negative}
\int_{\R^2} n(\ln n)_-=-\int_{D_1} n\ln n-\int_{D_2} n\ln n\leq
C\int_{D_1}\sqrt{n}+\int_{D_2} n\wx\leq C\int_{\R^2}
e^{-\frac{\abs{x}}{2}}+\int_{\R^2}n\wx.
\end{align}
This deduces the estimate \eqref{log-estimate}. Next, integrating
\eqref{estimate-weak-solution} in time $t$, we get
\[
\int_{\R^2} n(\cdot, t) \ln n (\cdot, t) + \mu \norm{\nabla c(t)
}^2_{L^2}+\frac{\lambda_1}{2} \norm{u(t)}^2_{L^2}+\lambda_1
\norm{n(t) \phi}_{L^1}
\]
%\[
%\le \lambda_1 C_5 \| n_0 \|_{L^1}^\frac12 \| \nabla \sqrt{n}
%\|_{L^2} + \lambda_1 C_6 \| n_0 \|_{L^1}^\frac12 \| \nabla \sqrt{n}
%\|_{L^2} \| \nabla c\|_{L^2}
%\]
\begin{equation}\label{log-est-100}
+\int_0^t\int_{\R^2} 2 \bke{|\nabla \sqrt{n}|^2+\frac{\mu}{2}
|\Delta c|^2 +\frac{\lambda_1}{2} |\nabla u |^2} dxd\tau\le
C_7+C_8t+C_9 \int_0^t\| \nabla c\|_{L^2}^2dxd\tau,
\end{equation}
where $C_7=\int_{\R^2} n_0 \ln n_0 + \mu \norm{\nabla c_0
}^2_{L^2}+\frac{\lambda_1}{2} \norm{u_0}^2_{L^2}+\lambda_1 \norm{n_0
\phi}_{L^1}$. Remembering \eqref{log-estimate}, we compute
\begin{equation}\label{xn}
\ddt\int_{\R^2} \wx n dx = \int_{\R^2} nu \na \wx dx + \int_{\R^2} n
\Del \wx dx + \int_{\R^2} \chi(c) n \na c \na \wx dx.
\end{equation}
The term   $\int_{\R^2} nu \na \wx dx $ is bounded as follows:
\[
\abs{\int_{\R^2} nu \na \wx dx }\leq \|\sqrt n
\|^2_{L^4}\|u\|_{L^2}\le \frac 12 \|\na \sqrt n\|_{L^2}^2 +
C\|n_0\|_{L^1}\|u\|_{L^2}^2.
\]
Noting that $\abs{\na \wx}+\abs{\Del \wx}\leq C$, we get
\[
\abs{\int_{\R^2} n \Del \wx dx} + \abs{\int_{\R^2} \chi(c) n \na c
\na \wx dx}%\leq C+C\norm{n}_{L^2}\norm{\nabla c}_{L^2}
\leq C+C\norm{\nabla \sqrt{n}}_{L^2}\norm{\nabla c}_{L^2},
\]
where we used that $\norm{n}_{L^2}\leq
C\norm{n_0}^{\frac{1}{2}}_{L^1}\norm{\nabla \sqrt{n}}_{L^2}$.
In summary, we obtain
\begin{align}\label{equ33}
\ddt \int_{\R^2} \wx n dx \le \delta \|\na \sqrt n \|_{L^2}^2 %+ \frac 12 \|\na u\|^2_{L^2}
+ C\|u\|_{L^2}^2 + C_{\delta}\norm{\nabla c}^2_{L^2} +C,
\end{align}
where $\delta$ is sufficiently small, which will be specified later.
Therefore, integrating \eqref{equ33} in time,
\begin{equation}\label{log-est-200}
\int_{\R^2} \wx n(\cdot,t) dx \le \int_{\R^2} \wx n_0 dx+\delta \int_0^t\|\na
\sqrt n \|_{L^2}^2 + C\int_0^t\|u\|_{L^2}^2 +
C_{\delta}\int_0^t\norm{\nabla c}^2_{L^2} +Ct.
\end{equation}
Now adding $2\int n(\ln n)_-$ to both sides of \eqref{log-est-100},
we obtain
\[
\int_{\R^2} n(\cdot, t) \abs{\ln n (\cdot, t)} + \mu \norm{\nabla
c(t) }^2_{L^2}+\frac{\lambda_1}{2} \norm{u(t)}^2_{L^2}+\lambda_1
\norm{n(t) \phi}_{L^1}
\]
\[
+\int_0^t\int_{\R^2} \bke{|\nabla \sqrt{n}|^2+\mu |\Delta c|^2
+\lambda_1 |\nabla u |^2} dxd\tau
\]
\begin{equation}\label{log-est-300}
\le C+Ct+C\int_0^t\| \nabla c\|_{L^2}^2dxd\tau+C\int_0^t\|
u\|_{L^2}^2dxd\tau,
\end{equation}
where $\delta$ in \eqref{equ33} is so small that term
$\int_0^t\|\na \sqrt n \|_{L^2}^2$ is absorbed to the left hand side of \eqref{log-est-100}. Since
\eqref{log-est-300} holds for all $t$ until the maximal time of
existence, due to Gronwall's inequality, we obtain $n \abs{\ln n}\in
L^{\infty}(0,\, T; L^2(\R^2))$. Moreover, again via the inequality
\eqref{log-est-300}, we deduce \eqref{n-eq-est}-\eqref{u-eq-est}.

We note that from the blow-up criterion in two dimensions in Theorem
\ref{Theorem3}, it suffices to show that $\nabla c \in L^2(0,\, T;
L^{\infty}(\R^2))$ for global existence of smooth solutions in
$\R^2$. We first consider the vorticity equation of velocity fields.
Taking curl, we have
\[
\pa_t \omega + (u \cdot \na ) \omega -\Delta \omega = -\na^{\perp} n
\cdot \na \phi,
\]
where $\na^{\perp}=(-\partial_2, \partial_1)$. If we multiply
$\omega$ to both sides of the above equation and integrate over
$\R^2$, then we have
\[
\frac12 \frac{d}{dt} \| \omega \|_{L^2}^2 + \| \na \omega \|_{L^2}^2
=\int_{\R^2} n \na \phi \na^{\perp} \omega  dx \le C \| n \|_{L^2}
\| \na^{\perp} \omega \|_{L^2}.
\]
Hence, we have
\[
\| \omega \|_{L^{\infty}(0,T; L^2)}^2 +\| \na \omega \|_{L^2(0, T;
L^2)}^2 \le C \| n \|_{L^2(0,T; L^2)}^2.
\]
Since $\| n \|_{L^2} \le \| \sqrt{n} \|_{L^4}^2 \le C \| \sqrt{n}
\|_{L^2} \| \na \sqrt{n} \|_{L^2}$, we have $\omega \in
L^{\infty}(0, T; L^2) \cap L^2(0, T; H^1)$. Next we consider the
equation of $n$. Multiplying $n$ and integrating over $\R^2$, we
have
\[
\frac12 \frac{d}{dt} \| n \|_{L^2}^2+ \| \na n \|_{L^2}^2 =
\int_{\R^2} \chi(c) n \na c \na n dx
\]
\[
=-\frac12 \int_{\R^2} \nabla \cdot ( \chi(c) \nabla c)n^2 dx \leq C \int_{\R^2} |\nabla^2 c| n^2 dx +C \int_{\R^2} |\nabla c|^2 n^2 dx
\]
\[
\leq C (\norm{\na^2 c}_{L^2}+\norm{\na c}^2_{L^4})\norm{n}^2_{L^4}\leq C
\norm{\na^2 c}_{L^2}\norm{n}_{L^2}\norm{\na n}_{L^2},
\]
where we used that  $\chi$ is $C^1$ and $c \in L^{\infty}(0,\,
\infty; L^{\infty})$, i.e., $\chi(c)$ and $\chi'(c)$ are bounded.
%Note that $ \| n \|_{L^{2+\alpha}} \le C\| n
%\|_{L^2}^{\frac{\alpha}{2+\alpha}} \| \na n
%\|_{L^2}^{\frac{2}{2+\alpha}}$ and $\| \na c
%\|_{L^{\frac{2(2+\alpha)}{\alpha}}} \le C \| \na c
%\|_{L^2}^{\frac{2}{2+\alpha}} \| \Delta c
%\|_{L^2}^{\frac{\alpha}{2+\alpha}}.$ Thus we have
%\[
%\| n \|_{L^{2+\alpha}} \| \na n \|_{L^2} \| \na c
%\|_{L^{\frac{2(2+\alpha)}{\alpha}}} \le C \| n \|_{L^2}^2 \| \na
%c \|_{L^2}^{\frac{4}{\alpha}}\| \Delta c\|_{L^2}^2+\frac14 \| \na
%n \|_{L^2}^2.
%\]
Due to Young's inequality, we have
\[
\frac{d}{dt} \| n \|_{L^2}^2+ \| \na n \|_{L^2}^2 \le C\| n
\|_{L^2}^2 \| \na^2 c\|_{L^2}^2.
\]
Therefore, via Gronwall's inequality, we have $n \in L^{\infty}(0,\,
T; L^2) \cap L^2(0, \, T; H^1)$. Multiplying $\Delta^2 c$ to
both sides of the equation of $c$ and integrating over $\R^2$, we
have
\[\frac12 \frac{d}{dt} \| \Delta
c \|_{L^2}^2 + \| \na \Delta c \|_{L^2}^2 \le \| \nabla u\|_{L^4} \|
\nabla c\|_{L^4} \| \nabla \Delta c \|_{L^2} +\|u \|_{L^{\infty}} \|
\nabla^2 c \|_{L^2} \| \nabla \Delta c \|_{L^2}- \int_{\R^2} k(c) n
\Delta^2 c dx.
\]
We note that the last term above is controlled as follows:
\[
\left|\int_{\R^2} k(c) n \Delta^2 cdx \right| \le \left|\int_{\R^2}
k'(c) \na c \cdot (\na \Delta c)n dx\right| + \left|\int_{\R^2} k(c)
\na n \cdot(\na \Delta c) dx\right|
\]
\[
\le C \| \na \Delta c \|_{L^2} \| n \|_{L^4} \| \na c
\|_{L^4}+C\|\nabla n \|_{L^2}\|\na \Delta c\|_{L^2}
\]

\[
\le \epsilon \| \na \Delta c \|_{L^2}^2 +C \| n \|_{L^4}^2 \| \na c
\|_{L^4}^2+C\| \nabla n \|_{L^2}^2.
\]
Hence, we have
\[
\frac12 \frac{d}{dt} \| \Delta c \|_{L^2}^2 + \| \na \Delta c
\|_{L^2}^2 \le C \| \na u \|_{L^4}^2 \| \na c
\|_{L^4}^2+C\|u\|_{L^{\infty}}^2\|\Delta c\|_{L^2}^2+ C \| n
\|_{L^4}^2 \| \na c \|_{L^4}^2+C\| \nabla n \|_{L^2}^2.
\]
\[
\leq C\norm{\omega}_{L^2}\norm{\nabla \omega}_{L^2}\norm{\nabla
c}_{L^2}\norm{\nabla^2 c}_{L^2}+C\norm{\nabla
\omega}^2_{L^2}\norm{\Delta c}^2_{L^2}
\]
\[
+C\norm{n}_{L^2}\norm{\nabla n}_{L^2}\norm{\nabla
c}_{L^2}\norm{\nabla^2 c}_{L^2}+C\norm{\nabla n}_{L^2}^2.
\]
Gronwall's inequality gives  $c \in L^{\infty}(0,\, T; H^2) \cap
L^2(0,\, T; H^3)$, which implies via embedding that $\nabla c\in
L^2(0,\, T; L^{\infty})$. This completes the proof.
\end{pfthm3}

\section{Global weak solution in three dimensions}
In this section we will show the global existence of the weak
solutions for \eqref{KSNS} in three dimensions. We start with
notations. $H_0^1(\bbr^3)$ is used to indicate the closure of
compactly supported smooth functions in $H^1(\bbr^3)$ and
$H^{-1}(\bbr^3)$ means the dual space of $H_0^1(\bbr^3)$. We also
introduce the function spaces $ \calV(\bbr^3), \calVs(\bbr^3),
\calH$ defined as follows:
\[
\calV (\bbr^3)= \{ u= (u_1, u_2, u_3)\,|\, u_i \in
H_0^1(\bbr^3)\},\qquad \calVs(\bbr^3)= \{  u \in \calV(\bbr^3)\, |
\,{\rm{div}}\, u = 0\},
\]
\[
\calH =  \mbox{ the closure of } \calVs(\bbr^3) \mbox{ in }
(L^2(\bbr^3))^3.
\]
The dual space of $\calV(\bbr^3)$ is denoted by $\calV'(\bbr^3)=
\{u=(u_1, u_2, u_3)\,\,|\,\, u_i \in H^{-1}(\bbr^3) \}$. The
duality $\langle w, v\rangle$ for $w \in \calV'(\bbr^3), v\in
\calV(\bbr^3)$  is, as usual, given as $\langle w,v \rangle =
\sum_{i=1}^3 \langle w_i,\, v_i \rangle _{H^{-1} \times H^1_0}$ and
we denote $\calV_{\sigma}^\circ(\bbr^3) = \{w\in \calV'(\bbr^3)\,|
\,\langle w, v \rangle = 0 \mbox{ for all } v\in
\calV_{\sigma}(\bbr^3)\}$.

Next, we define the notion of a weak solution for the system
\eqref{KSNS}.
\begin{defn}\label{defweak}
Let $0<T\leq \infty$. A triple $(n,\, c,\, u)$ is called a weak
solution to the Cauchy problem \eqref{KSNS} in $\R^3\times [0,T)$ if
the following conditions are satisfied:
\begin{itemize}
\item[(a)] The functions $n$ and $c$ are non-negative and $(n,\, c,\, u)$
satisfy
\[
n(1+|x| +|\ln n|) \in L^{\infty}(0, \, T; L^{1} (\R^3)),\quad \nabla
\sqrt{n} \in L^{2}(0, \, T; L^{2} (\R^3)),
\]
\[
c \in L^{\infty} (0,\, T; L^{1}(\R^3)\cap L^{\infty}(\R^3) \cap
H^{1} (\R^3)),\quad \nabla c \in L^{2} (0, \, T; L^{2}(\R^3)),
\]
\[
u \in L^{\infty}(0,\, T; L^{2}(\R^3)) ,\quad \nabla u \in L^{2}(0,\,
T; L^{2}(\R^{3})).
\]
\item[(b)] The functions $n, c$, and $u$ solve the
chemotaxis-fluid equations \eqref{KSNS} in the sense of
distributions, namely for any $\Psi \in  C^1([0,T];
(C_c^{\infty}(\bbr^3))^3) $ with $\na\cdot \Psi =0$
\[
\int_{\R^3}(u\cdot\Psi)(\cdot, T) +\int_{0}^{T} \int_{\R^{3}} u
\cdot (\partial_t \Psi + \Delta \Psi)+\int_{0}^{T} \int_{\R^{3}}u
\otimes u : \na \Psi
\]
\[
- \int_0^{\infty} \int_{\R^3} n \nabla \phi \cdot \Psi + \int_{\R^3}
u_0 \cdot \Psi(0, x) =0,
\]
where $u \otimes u:\na\Psi = \sum_{j,k=1}^3 u^j u^k \pa_j \Psi^k$
and
\[
\int_0^{\infty} \int_{\bbr^3 }n(\pa_t \varphi +\Delta \varphi)
+\int_0^{\infty} \int_{\R^3} nu \cdot \na \varphi +\int_0^{\infty}
\int_{\R^3} \chi (c) n \na c \cdot \na \varphi
 +\int_{\R^3} n_0 (x) \varphi (0, x) =0,
\]
\[
\int_0^{\infty} \int_{\bbr^3 }c(\pa_t \varphi +\Delta \varphi)
+\int_0^{\infty} \int_{\R^3} cu \cdot \na \varphi -\int_0^{\infty}
\int_{\R^3} k (c) n \varphi +\int_{\R^3} c_0 (x) \varphi (0, x) =0
\]
for any $\varphi \in  C^1([0,T]; (C_c^{\infty}(\bbr^3)))$ with
$\varphi(\cdot, T)=0$.
\item[(c)] The functions $n$, $c$ and $u$ satisfy the following energy
inequality:
\[\intd (\frac{|u|^2}{2} + n\phi + n |\ln n|  + \frac{|\na c|^2}{2} + \wx n) dx+
\int_0^T  \| \na u\|_{L^2}^2 +  \|\na \sqrt n \|_{L^2}^2 + \|\Del
c\|_{L^2}^2 dt  \le C, \]
with  $C=C(T, \|\chi(c)\|_{L^{\infty}}, \|\langle x \rangle
n_0\|_{L^1},  \| \na c_0\|_{L^2},
 \|{n_0}|\ln {n_0}|\|_{L^1},  \|\Del \phi\|_{L^{\infty}}, \|\na \phi\|_{L^{\infty}},
 \|\phi\|_{L^{\infty}})$.
\end{itemize}

\end{defn}
Now we compute a priori estimate of an energy inequality under the
Assumption {\bf{(AA)}} and~{\bf{(B)}}.
%What it follows, we  assume that the coefficient functions $\chi(c)$, $k(c)$ satisfy the condition
%\[ \chi(c) - \mu k(c)=0,\] and (B).
%We remark that formal computations yield the following inequalities
%for smooth solutions with sufficient integrability. \\
%\indent
We note first, by maximum principle, that
\[
n(t,x) \ge 0, \quad c(t,x) \ge 0, \quad \|c(t)\|_{L^p} \le
\|c_0\|_{L^p} \quad\mbox{ for }\,\, t \ge 0, \,\, 1\le p \le \infty.
\]
It is straightforward that $\|n(t)\|_{L^1} = \|n_0\|_{L^1}$ for
$t\ge 0$ and
\begin{equation}\label{eq1}
\ddt\left(\intd  \frac {|u|^2}{2} dx + \intd n\phi dx \right) +
\intd |\na u|^2 dx  = \intd n \Del \phi dx + \intd \chi(c) n \nabla
c \na \phi dx,
\end{equation}
\begin{equation}\label{eq1-2}
\ddt \intd n\ln n dx + \intd \frac{|\na n|^2}{n} dx + \intd \chi'(c)
|\na c|^2 n dx = -\intd \chi(c)\Delta c n dx,
\end{equation}
\begin{equation}\label{eq1-3}
\ddt \intd |\na c|^2 dx + \intd |\Del c|^2 dx = \intd k(c)\Delta c n
dx + \sum_{i,j=1}^{3}\intd c\pa_i\pa_j c \pa_i u_j dx.
\end{equation}
Multiplying $\mu$ to the last equation \eqref{eq1-3} and adding it
to the second equation \eqref{eq1-2}, we have
\[
\ddt\left(\intd n \ln n  dx +  \mu |\na c|^2 dx \right) + \intd |\na
\sqrt n|^2 dx + \mu\intd |\Del c|^2 dx + \int \chi'(c)|\na c|^2 n dx
\]
\begin{equation}\label{eq2}
\le  -\intd \underbrace{(\chi(c)- \mu k(c))}_{=0}\Delta c n dx + \mu
\|c_0\|_{L^{\infty}}\|\na u\|_{L^2}\|\Del c\|_{L^2}\le
\frac{C_1}{2}\|\na u\|_{L^2}^2 + \frac{\mu}{4} \|\Del c\|_{L^2}^2
\end{equation}
for some $C_1$, which can be taken bigger than $1$, i.e. $C_1> 1$.
Also it holds that
\begin{equation} \label{xn}
\ddt\intd \wx n dx = \intd nu \na \wx dx + \intd n \Del \wx dx +
\intd \chi(c) n \na c \na \wx dx.
\end{equation}
Since the term  $\intd nu \na \wx dx $ is bounded as follows:
\[
\|n\|_{L^{\frac 65}}\|u\|_{L^6} \le C\|n\|_{L^1}^{\frac 34} \|\na
\sqrt n\|_{L^2}^{\frac 12}\|\na u\|_{L^2} \le \frac 12  \|\na \sqrt n
\|^2_{L^2} + \frac 12\| \na u\|^2_{L^2} + C(\|n_0\|_{L^1}),
\]
we can have
\begin{align}\label{eq3}
\ddt \intd \wx n dx \le \frac 12 \|\na \sqrt n \|_{L^2}^2 + \frac 12
\|\na u\|^2_{L^2} + \intd \chi(c) n \na c \na \wx dx +C.
\end{align}
We estimate the term $\intd \chi(c) n \na c \na \wx dx$ similarly as
above.
\[
\intd  \chi(c) n \na c \na \wx dx  \le  C\| \sqrt n\|_{L^3}^2\|\na
c\|_{L^3} \le C\|\sqrt n\|_{L^2}\|\na \sqrt n\|_{L^2}\|\na
c\|_{L^2}^{\frac 12}\|\Del c\|_{L^2}^{\frac 12}
\]
\begin{equation}\label{eq33}
\le C\|\na c\|_{L^2}\|\Del c\|_{L^2} + \frac 14 \|\na \sqrt
n\|_{L^2}^2\le C\|\na c\|^2 _{L^2}+ \frac 14 \|\Del c\|^2_{L^2} +
 \frac 14 \|\na \sqrt n\|_{L^2}^2.
\end{equation}
Multiplying $C_1$ to \eqref{eq1} and adding it together with
\eqref{eq2} and \eqref{eq3}, we have
\[
\ddt\left( \intd C_1(\frac{|u|^2}{2} + n\phi ) + n\ln n + \frac{|\na
c|^2}{2} + \wx n dx \right)
\]
\begin{equation}\label{add}
+ \frac{C_1-1}{2} \| \na u\|_{L^2}^2 + \frac 14 \|\na \sqrt n
\|_{L^2}^2 + \frac 14\|\Del c\|_{L^2}^2\le C (\| \na c\|_{L^2}^2 +
\|u\|^2_{L^2} ) + C.
\end{equation}
Then, by Gronwall's inequality, we have
\[
\intd (\frac{|u|^2}{2} + n\phi + n\ln n + \frac{ |\na c|^2}{2} + \wx n) dx
\]
\begin{equation}\label{ineq}
+ \int_0^T  \| \na u\|_{L^2}^2 +  \|\na \sqrt n \|_{L^2}^2 + \|\Del
c\|_{L^2}^2 dt  \le C,
\end{equation}
where $C(T, \|\chi(c)\|_{L^{\infty}},  \|n_0\|_{L^1}, \| \wx n_0\|_{L^1},\|\Del
\phi\|_{L^{\infty}}, \|\na \phi\|_{L^{\infty}})$. By same reasoning
for treating $ n(\ln n)_-$ term in \eqref{negative}, it follows that
\begin{equation}\label{ineq1}
\intd (|u|^2 + n\phi + n\abs{\ln n} + |\na c|^2 + \wx n) dx+
\int_0^T  \| \na u\|_{L^2}^2 +  \|\na \sqrt n \|_{L^2}^2 + \|\Del
c\|_{L^2}^2 dt  \le C.
\end{equation}
Streamline of constructing global weak solutions, as in usual steps
for the Navier-Stokes equations, is the following:\\
\textit{
$\cdot$ regularizing the system for which we prove the existence of smooth solutions\\
$\cdot$ finding uniform estimates for the solutions of the regularized system\\
$\cdot$ passing to the limit on the regularized parameters.\\}
\begin{subsection}{Regularization}
In this subsection, we intend to construct approximate solutions of
the system.
\indent
For the incompressible Navier-Stokes equations defined on a general
bounded domain, the global weak solutions are constructed by using
the spectral projections $(P_{k})_{k\in \mathbb{Z}}$, associated to
the inhomogeneous Stokes operator (\cite[Chapter 2]{CDGG}). A number of useful properties of
the family $(P_k)_{k\in \mathbb{Z}}$ are listed as follows:
 For any $ u\in \mathcal{H}(\Omega),$
\begin{align}
& P_k P_{k'} u = P_{min(k,k')}u ,\qquad
\lim_{k\to\infty} \|P_k u - u\|_{\mathcal{H}(\Omega)} =0, \label{ext}\\ %\label{app}
& \| \na P_k u\|_{L^2(\Omega)} \le \sqrt k\|u\|_{L^2(\Omega)},
\qquad \| \Del P_k u\|_{L^2(\Omega)} \le
k\|u\|_{L^2(\Omega)}\label{smoothing},\\
& \| (1-P_k)u\|_{L^2} \le \frac{1}{\sqrt k}\|u\|_{\mathcal
V_{\sigma}}. \label{difference}
\end{align}
In particular, \eqref{smoothing} implies $P_k u \in
L^{\infty}(\Omega)$ for $u\in L^2(\Omega)$ in three dimensions.
%{\tt Jihoon: $\Omega=\bbr^3$ in our case?}
\begin{defn}
The bilinear map Q is defined by
\begin{align*}
&Q: \calV \times \calV \to \calV',\\
& (u,v) \mapsto -{\rm{div }}\,(u\otimes v).
\end{align*}
%In what follows, for convenience, we denote $P_k Q(u,u)$ by
%$F_k(u)$.
\end{defn}
 From now on  we denote by $\mathcal{H}_k(\bbr^3)$  the space $P_k \calH$.
We regularize   \eqref{KSNS} by a frequency cut-off operator $P_k$ and a mollifier $\sigma^{\ep}$:
 \begin{equation}\label{regular} \left\{
\begin{array}{ll}
\partial_t {n^{k,\ep}}(t) = - {u^{k,\ep}} \cdot \nabla  {n^{k,\ep}} + \Delta {n^{k,\ep}} -
\nabla\cdot ({n^{k,\ep}}[(\chi ({c^{k,\ep}})  \nabla {c^{k,\ep}})\ast \sigma^{\ep}]),\\
\vspace{-3mm}\\
\partial_t {c^{k,\ep}}(t)= - {u^{k,\ep}} \cdot \nabla {c^{k,\ep}}+\Delta {c^{k,\ep}}
-k({c^{k,\ep}}) ({n^{k,\ep}}\ast \sigma^{\ep}),\\
\vspace{-3mm}\\
\partial_t {u^{k,\ep}}(t)= -P_kQ({u^{k,\ep}},\, {u^{k,\ep}})+P_k\Delta {u^{k,\ep}}
-P_k({n^{k,\ep}} \nabla \phi),
\end{array}
\right.
\end{equation}
with initial data
\[
(n_0^{k,\ep}, c_0^{k,\ep}, u_0^{k, \ep})=(n_0\ast\sigma^{\ep},
c_0\ast\sigma^{\ep}, P_k u_0\ast \sigma^{\ep}),
\]
where $n_0, c_0, u_0$ is the initial data of \eqref{KSNS} satisfying
the condition \eqref{weakdata} in Theorem \ref{Theorem5}. The mollifier is
defined as usual such that
$\sigma^{\ep}(x)=\ep^{-3}\sigma(\ep^{-1}x)$ for $\sigma \in
C_0^{\infty}(\bbr^3)$. Apart from the frequency cut-off the
regularization is same one for a chemotaxis-fluid model studied in
\cite{Liu-Lorz}.
%Mollifying a vector fields is understood coordinatewise.
Repeating similar arguments in Theorem \ref{Theorem2}, we obtain the
local solution of \eqref{KSNS} in the class
\begin{align}\label{spaces}\begin{aligned}
&{n^{k,\ep}}\in L^{\infty}(0,T; H^{m-1}(\bbr^3))\cap L^2(0,T;H^m(\bbr^3))\\
&{c^{k,\ep}}\in L^{\infty}(0,T; H^{m-1}(\bbr^3))\cap L^2(0,T;H^m(\bbr^3))\\
&{u^{k,\ep}}\in L^{\infty}(0,T; H^{m-1}(\bbr^3)\cap \mathcal{H}_k(\bbr^3))
\cap L^2(0,T;H^m(\bbr^3))
\end{aligned}
\end{align}
for some time $T$ and for all $m>3$. It turns out that due to the
regularization of nonlinear terms and  smoothing properties of $P_k$
(see \eqref{smoothing}), the local solution of \eqref{KSNS} can be
extended up to infinite time.
\begin{proposition}\label{prop1}
The regularized system \eqref{regular} has the unique global
solution $(n^{k,\ep}, c^{k,\ep}, u^{k,\ep})$ in a class
\eqref{spaces} for any time $T<\infty$.
% Moreover $T$ is arbitrary. Thus the solution is global.
\end{proposition}
Before presenting the proof we observe that the  approximating
solution $({n^{k,\ep}},{c^{k,\ep}},{u^{k,\ep}})$ of  \eqref{regular}
satisfies an energy inequality.
\begin{proposition}\label{prop2}
The  solution $({n^{k,\ep}},{c^{k,\ep}},{u^{k,\ep}})$ of
\eqref{regular} satisfies the following inequality.
\begin{align}\label{appineq}
\begin{aligned}
&\intd (\frac{|{u^{k,\ep}}|^2}{2} + {n^{k,\ep}}\phi ) +
{n^{k,\ep}}|\ln {n^{k,\ep}}| + \frac{|\na {c^{k,\ep}}|^2}{2} +
\wx {n^{k,\ep}} dx   \\
& \qquad + \int_0^T   \| \na {u^{k,\ep}}\|_{L^2}^2 +  \|\na
\sqrt {n^{k,\ep}} \|_{L^2}^2 + \|\Del {c^{k,\ep}}\|_{L^2}^2 dt  \le
C,
\end{aligned}
\end{align}
where $C=C\bke{T, \|\chi(c)\|_{L^{\infty}}, \|\langle x \rangle
n_0\|_{L^1},  \| \na c_0\|_{L^2},
 \|{n_0}|\ln {n_0}|\|_{L^1},  \|\Del \phi\|_{L^{\infty}}, \|\na \phi\|_{L^{\infty}}, \|\phi\|_{L^{\infty}}}$.
\end{proposition}
\begin{proof}
We note that  the same cancellation  as in \eqref{eq2} holds for the
regularized system \eqref{regular}, hence $({n^{k,\ep}},
{c^{k,\ep}}, {u^{k,\ep}})$ satisfying \eqref{spaces} satisfy the
energy inequalities \eqref{eq1}-\eqref{xn}.  Moreover  the following
moment bound holds by similar estimates as \eqref{eq3},
\eqref{eq33},
\begin{align*}
\ddt \intd \wx {n^{k,\ep}} dx &= \intd {n^{k,\ep}} {u^{k,\ep}}\na
\wx dx + \intd {n^{k,\ep}} \Del \wx dx + \intd
{n^{k,\ep}}[(\chi({c^{k,\ep}})\na {c^{k,\ep}})\ast \sigma^{\ep} ]\na
\wx dx \\& \le  C\|u^{k,\ep}\|_{L^2}\| \na u^{k,\ep}\|_{L^2}+ C
\|\na c^{k,\ep}  \|_{L^2}\|\Del c^{k,\ep} \|_{L^2} + \frac 12 \| \na
\sqrt{n^{k,\ep}}\|_{L^2}^2 + \| n_0\|_{L^1}.
\end{align*}
Then we  have  \eqref{appineq} with T depending on $\|\na
{c_0^{k,\ep}}\|_{L^2} $, $\| \wx {n_0^{k,\ep}} \|_{L^1}$,
$\|{n_0^{k,\ep}}|\ln {n_0^{k,\ep}}|\|_{L^1}$. It is immediate to
have
\[
\|\na {c_0^{k,\ep}}\|_{L^2} + \| \wx {n_0^{k,\ep}} \|_{L^1} \le
\|\na {c_0}\|_{L^2}+ \| \wx {n_0} \|_{L^1}.
\]
Note that $x\ln x$ is convex and $d\mu = \sigma^{\ep}(y) dy$
provide a probability measure. Then by Jensen's inequality, we have
\[
{n_0^{k,\ep}}(\ln {n_0^{k,\ep}})_+ \le (n_0(\ln n_0)_+)\ast
\sigma^{\ep}.
\]
Integrating the above in $x$ and observing that
$\lim_{\ep\rightarrow 0} \norm{(n_0|\ln n_0|)\ast
\sigma^{\ep}}_{L^1}=\norm{n_0|\ln n_0|}_{L^1}$, we have
\begin{equation}\label{jensen}
\|{n_0^{k,\ep}}(\ln {n_0^{k,\ep}})_+ \|_{L^1} \le \|{n_0}|\ln
{n_0}|\|_{L^1}.
\end{equation}
For the $\|{n_0^{k,\ep}}(\ln {n_0^{k,\ep}})_- \|_{L^1}$, proceeding
similarly as \eqref{negative}, we have
\[
\|{n_0^{k,\ep}}(\ln {n_0^{k,\ep}})_- \|_{L^1} \le C+ \int_{\R^3}
n^{k, \ep} \wx dx \le C\bke{1+ \int_{\R^3} n\wx dx},
\]
from which we deduce the proposition.
% \{ Note that $x\ln|x|$ is convex and $d\mu = \sigma^{\ep}(y) dy$
% provide a probability measure. Then by Jensen's inequality, we have
% \[
% {n_0^{k,\ep}}|\ln {n_0^{k,\ep}}| \le (n_0|\ln n_0|)\ast
% \sigma^{\ep}.
% \]
% Integrating the above in $x$ and observing that
% $\lim_{\ep\rightarrow 0} \norm{(n_0|\ln n_0|)\ast
% \sigma^{\ep}}_{L^1}=\norm{n_0|\ln n_0|}_{L^1}$, we have
% \begin{equation}\label{jensen}
% \|{n_0^{k,\ep}}(\ln {n_0^{k,\ep}})_+ \|_{L^1} \le \|{n_0}|\ln
% {n_0}|\|_{L^1},
% \end{equation}
% which concludes the proposition.\}
\end{proof}
Now we give the proof of Proposition \ref{prop1}.\\
{\bf{Proof of Proposition \ref{prop1}}}\quad We first observe that
the regularity criterion in Theorem \ref{Theorem3} hold true for the
system \eqref{regular}. Since its verification is tedious repetition
of that of Theorem \ref{Theorem3}, we omit its details. If we
consider the second equation of \eqref{regular}, then we have the
following energy estimates.
\[
\frac12 \frac{d}{dt} \| c^{k, \ep} \|_{H^2}^2 + \|\na
c^{k,\ep}\|_{H^2}^2 \leq C\| \na u^{k, \ep}\|_{L^3}^2 \|c^{k,
\ep}\|_{L^6}^2 + C\|\na c^{k, \ep}\|_{L^6}^2 \|n^{k,
\ep}\|_{L^1}^2+\frac12 \|\na c^{k, \ep}\|_{H^2}^2
\]
\[
\leq C k^{\frac12} \|u^{k, \ep}\|_{L^2}^2\|c^{k, \ep}\|_{H^2}^2+C
\|n^{k, \ep}_0\|_{L^1}^2\| c^{k, \ep}\|_{H^2}^2+\frac12 \|\na c^{k,
\ep}\|_{H^2}^2.
\]
By using Gronwall's inequality, we have $\norm{\na c^{k,
\ep}}_{L^{\infty}_{x}L^{2}_t}< \infty$. Since $\|u^{k,
\ep}(t)\|_{L^2}$ is bounded and  $\|\na u^{k, \ep} \|_{L^2} \le C
\sqrt k \|u^{k, \ep}\|_{L^2}$, we can also demonstrate that the
Serrin condition in Theorem \ref{Theorem3} is satisfied for $u^{k,
\ep}$. This completes the proof. \qed
\end{subsection}

\begin{subsection}{Global weak solutions}
In this subsection, we give the proof of Theorem \ref{Theorem5}.
%\\{\tt \{$L^1$ convergence of $n_0^{l, \ep}$  and  $L^1$ uniform boundedness of
%$\wx n_0^{l,\ep} ,n_0^{l,\ep} |\ln n_0^{l,\ep}| $ are enough.\}}
\begin{pfthm4}
We consider an approximating sequence $(n_0^{l,\ep}, c_0^{l,\ep},
u_0^{l, \ep})$ to $(n_0, c_0, u_0)$. Note that
\[
\int_{\R^3} |n_0^{l,\ep}-n_0| +|\na c_0^{l,\ep} -\na c_0| dx + \int_{\R^3} |u_0^{l,\ep} - u|^2 dx \rightarrow 0.
,\] and
\[
\int_{\R^3}\wx n_0^{l,\ep}  dx +
\int_{\R^3} n_0^{l,\ep} |\ln n_0^{l,\ep}| \le C \int _{\R^3} \wx n_0 dx +
\int _{\R^3}  n_0|\ln n_0| dx  +C.
\]
We denote by $({n^{l,\ep}}, {c^{l,\ep}}, {u^{l,\ep}})$ the approximating solution constructed
in the previous section for the system \eqref{regular} with
initial data $({n^{l}}(0,\cdot), {c^{l}}(0, \cdot)) = (n_0^l(0,\cdot), c_0^l(0, \cdot))$ and
${u^{l}}(0,\cdot) = P_l u_0(\cdot)$.
%By the Maximum principle and the energy inequality \eqref{appineq}
Several uniform
estimates hold for the approximating solutions:
\begin{align}
\|{c^{l,\ep}}\|_\pq{\ify}{p} \le C  \mbox{ for } 1\le p\le \infty, \label{csup}\\
\|{c^{l,\ep}}\|_{L^{\ify}(0,T; H^1(\bbr^3))} + \| \Del {c^{l,\ep}} \|_\td   \le C, \label{cb}\\
\|\sqrt {n^{l,\ep}} \|_\pq{\ify}{2} + \| \na \sqrt {n^{l,\ep}} \|_\td \le C, \label{nb} \\
\| {u^{l,\ep}}\|_\pq{\ify}{2} + \| \na {u^{l,\ep}} \|_\td \le C \label{ub}.
\end{align}
Then there exists subsequences $n^{l,\ep}, c^{l,\ep}, u^{l,\ep}$ and
some functions $ n,c,u $ such that
\[
\sqrt{n^{l,\ep}} \rightharpoonup \sqrt{n} \quad
L^{\infty}(0,T;L^2(\bbr^3))-{\rm{weak}}^*,
\]
\[
c^{l,\ep} \rightharpoonup  c  \quad  L^{\infty}(0,T;L^p(\bbr^3))
\cap L^{\infty}(0,T; H^1(\bbr^3))-{\rm{weak}}^*,
\]
\[
u^{l,\ep} \rightharpoonup u \quad
L^{\infty}(0,T;L^2(\bbr^3))-{\rm{weak}}^* \cap
L^2(0,T;\mathcal{V}_{\sigma}(\bbr^3))-{\rm{weak}}
\]
for $1\le p\le \infty$. Let us show that  $n,c,u$ is  a weak
solution in the sense of Definition \ref{defweak}. By
Gagliardo-Nirenberg inequality and \eqref{nb}, we have
\[
\intd |{n^{l,\ep}}|^p dx \le C\|n_0\|^{\frac{3-p}{2}}_{\lod}\| \na
\sqrt {{n^{l,\ep}}}\|_\ltd^{3(p-1)},
\]
and therefore,
\begin{equation}\label{nbb}
\|{n^{l,\ep}}\|_\pq{q}{p} < C(T),
\quad 1\le q\le \frac{2p}{3(p-1)}
\end{equation}
for $1\le p \le 3$. Some strong convergences are necessary. We note
that \eqref{nbb} implies the source term of the Navier Stokes
equation ${n^{l,\ep}}\na \phi$ is in
 $L^2([0,T]; \calVs'(\bbr^3))$ uniformly  with respect to $l$;
for any $w\in L^2([0,T]; \calVs(\bbr^3))$, it holds that
\begin{align*}
\int_0^T \int_{\bbr^3} P_l({n^{l,\ep}}\na \phi) w dxdt \le \|\na
\phi \|_{L^{\ify}(\bbr^3)}\| {n^{l,\ep}}\|_\pq{2}{\frac
65}\|w\|_\pq{2}{6}.
\end{align*}
It proves that $\pa_t {u^{l,\ep}} $ is uniformly bounded in
$L^2(0,T; \calVs'(\bbr^3))$.
 Note that $u_k$ is uniformly bounded in $L^{\infty}(0,T; \mathcal H(\bbr^3))) \cap
 L^2(0,T; \calVs(\bbr^3))$ due to \eqref{appineq}. Combining these facts and \eqref{smoothing}, \eqref{difference}
we  have  compactness result for $({u^{l,\ep}})$ (see  \cite[Proposition 2.7]{CDGG} for detailed proof):
there exists $u$ in $ L^2(0,T; \mathcal V_{\sigma}(\bbr^3))$ such that up to subsequence
\begin{equation}\label{cpt}
\lim_{l\to\infty,\ep\to 0} \int _0^T \int _K |{u^{l,\ep}}(t,x)- u(t,x)|^2 dx
dt=0,
\end{equation}
for any $T>0$ and compact subset $K$ of $\bbr^3$. In addition, for
$\Psi \in L^2([0, T]; \calV(\bbr^3))$ and $\Phi \in L^2([0,T]\times
\bbr^3)$
\begin{align}\label{lim1} \begin{aligned}
\lim_{l\to\infty,\ep\to 0} \int_0^T \int_{\R^3} \na u^{l,\ep}(t,x) \na
\Psi(t,x) dx dt& =
 \int_0^T \int_{\R^3} \na u(t,x) \na \Psi(t,x) dx dt, \\
 \lim_{l\to\infty,\ep\to 0} \int_0^T \int_{\R^3}  u^{l,\ep}(t,x) \Phi(t,x) dxdt &
 = \int_0^T \int_{\R^3} u(t,x) \Phi(t,x) dxdt.
 \end{aligned}
 \end{align}
 Furthermore, For any $\psi \in C^1(\bbr^+; \calVs(\bbr^3))$
 \begin{equation}\label{lim2}
 \lim_{l\to\infty,\ep\to 0} \sup_{t\in [0,T]}\left|\int_{\R^3} ({u^{l,\ep}}(t,x)-u(t,x)) \psi(t,x) dx\right| =0.
 \end{equation}
Applying a test function $\Psi$ in $C^1([0,T];\calVs(\bbr^3))$, we
obtain
\begin{align}\label{pass}\begin{aligned}
\ddt \langle {u^{l,\ep}}(t),\Psi(t) \rangle &= \langle  \Del {u^{l,\ep}} (t),P_l \Psi(t) \rangle +
\langle  Q({u^{l,\ep}}(t), {u^{l,\ep}}(t)), P_l\Psi(t) \rangle \\
& + \langle  ({n^{l,\ep}}\na \phi), P_l \Psi(t) \rangle +
\langle {u^{l,\ep}}(t), \ddt \Psi(t) \rangle.
\end{aligned}
\end{align}
Following the arguments in \cite{CDGG}, that is,  using
\eqref{cpt}-\eqref{lim2} and the fact
\begin{equation}\label{fact}
\lim_{l\to\infty} \sup_{t\in [0,T]} \| P_l \Psi(t)-
\Psi(t)\|_{\calV(\bbr^3)}=0,
\end{equation}
we can pass to the limit with respect to $l$  so that
\begin{align*}
&\int _{\bbr^3} u\cdot \Psi(T,x) dx + \int _0^T \int_{\bbr^3} (\na u: \na \Psi - u\otimes u : \na \Psi
- u\cdot \pa_t \Psi ) (s,x) dxds \\
&=  \int _{\bbr^3} u_0 (x) \Psi(0,x) dx + \lim_{l\to \infty,\ep\to 0}
\int_0^T \langle  {n^{l,\ep}}\na \phi , \Psi \rangle dt.
\end{align*}
For the strong convergence of $({n^{l,\ep}})$ we have $
\sqrt{{n^{l,\ep}}} \to \sqrt n $ strongly in $L^2_{loc}(\bbr^3)$ for
a.e. $t\in [0,T]$ by Sobolev embedding. Since $\|\sqrt{{n^{l,\ep}}}
(t)\|_\ltd$ is continuous in time, we redefine $n(t)$ such that  $
\| \sqrt{{n^{l,\ep}}}-  \sqrt n \|_\ltd \to 0 $ for all $t\in
[0,T]$. Then by \eqref{nbb} and Lebesgue Dominated convergence
theorem, it follows that
\begin{equation}\label{nconv}
\| {n^{l,\ep}}-n\|_{L^q(0,T; L^p_{loc}(\bbr^3))} \to 0, \qquad 1\le
q\le \frac{2p}{3(p-1)}
\end{equation}
for $1\le p \le 2$. For the convergence of $({c^{l,\ep}})$ we have
${c^{l,\ep}}(t) \to c(t)$ strongly in $L^2_{loc}(\bbr^3)$ for all
$t\in [0,T]$ and therefore,
\begin{equation}\label{cconv}
\| {c^{l,\ep}}-c\|_{L^p_{loc}((0,T)\times \bbr^3)} \to 0, \qquad
1\le p<\infty
\end{equation}
by the uniform boundedness \eqref{csup}. Moreover we have
\begin{equation}\label{deri}
\| \na {c^{l,\ep}} - \na c\|_{L^2(0,T; L^{p}_{loc}(\bbr^3))} \to 0,
\qquad 1 \le p < 6.
\end{equation}
By \eqref{cb}, $\|\na {c^{l,\ep}}\|_{L^2(0,T; H^1(\bbr^3))}$ is
uniformly bounded. For any $ \na g\in \pq{4}{2}$, we have
\begin{align*}
\int_0^T \intd \pa_t \na {c^{l,\ep}} g dxdt \le \int_0^T \intd
{u^{l,\ep}} \na {c^{l,\ep}} \na g + \Del {c^{l,\ep}} \na g
+ k({c^{l,\ep}})(n^{l,\ep}\ast \sigma_{\ep}) \na g dxdt.
\end{align*}
We estimate
\begin{align*}
\int_0^T \intd {u^{l,\ep}} \na {c^{l,\ep}} \na g dxdt &\le
C \int_0^T \|{u^{l,\ep}}\|_{L^6(\bbr^3)} \|\na {c^{l,\ep}}\|^{\frac 12}_\ltd \|\Delta {c^{l,\ep}}\|^{\frac 12}_\ltd \|\na g\|_\ltd dt \\
& \le C \int_0^T \| \na {u^{l,\ep}}\|_\ltd\|\Delta {c^{l,\ep}}\|^{\frac 12}_\ltd\|\na g\|_\ltd dt\\
&\le C\| \na {u^{l,\ep}}\|_\td\|\Delta {c^{l,\ep}}\|^{\frac 12}_\td\|\na g\|_\pq{4}{2},
\end{align*}
\begin{align*}
\int_0^T \intd k({c^{l,\ep}}){(n^{l,\ep}\ast \sigma_{\ep})} \na g
dxdt\le C\| {n^{l,\ep}}\|_\pq{\frac 43}{2} \|\na g\|_\pq{4}{2}.
\end{align*}
Thus we have $\pa_t {c^{l,\ep}} \in L^{\frac 43}(0,T;
H^{-1}(\bbr^3))$. The strong convergences \eqref{nconv}-\eqref{deri}
are enough to pass to the limit for nonlinear terms in the chemotaxis
part. For instance, testing a $\Psi \in C_c^{\infty}(\bbr^3)$  to
the worst nonlinear term $\na\cdot({n^{l,\ep}} (\chi({c^{l,\ep}})
\na {c^{l,\ep}})\ast \sigma_{\ep})$, we have
\begin{align*}
&\int_0^T \intd   \na \cdot (n^{l,\ep}[(\chi(c^{l,\ep}) \na c^{l,\ep})\ast \sigma_{\ep}]) \Psi
- \na \cdot (n\chi(c)\na c)\Psi dxdt\\
& =  \int_0^T \intd  (n^{l,\ep}-n)[(\chi({c^{l,\ep}})\na c^{l,\ep})\ast \sigma_{\ep}] \na \Psi dxdt\\
&+  \int _0^T\intd n [(\chi({c^{l,\ep}})\na c^{l,\ep})\ast \sigma_{\ep} -\chi(c)\na c] \na\Psi dxdt.
\end{align*}
The second integral is
\begin{align*}
&\int_0^T\intd [(n\na \Psi)\ast \sigma_{\ep} - n\Psi]\chi(c^{l,\ep})\na c^{\l, \ep}
+ n\na\Psi( \chi(c^{l,\ep})\na c^{l, \ep} - \chi(c)\na c) dxdt \\
&= \int_0^T \intd[(n\na \Psi)\ast \sigma_{\ep} - n\Psi]\chi(c^{l,\ep})\na c^{\l, \ep}
+ n\na\Psi ([ (\chi(c^{l,\ep})-\chi(c)]\na c^{l,\ep} + \chi(c)(\na c^{l,\ep} -\na c)) dxdt.
\end{align*}
The integrals  go to zero by the uniform estimates \eqref{csup}-\eqref{ub} and \eqref{nconv}-\eqref{deri} with the Lipschitz
continuous assumption on $\chi(\cdot)$.\\
Lastly, we consider the approximated energy inequality
\eqref{appineq} replacing $n^{\l, \ep} |\ln n^{l,\ep}|$ with $n^{\l,
\ep} \ln n^{l,\ep}$. Taking the limit and uing the convexity of
$x\ln x$ we deduce
\[\intd (\frac{|u|^2}{2} + n\phi + n \ln n  + \frac{|\na c|^2}{2} + \wx n) dx+
\int_0^T  \| \na u\|_{L^2}^2 +  \|\na \sqrt n \|_{L^2}^2 + \|\Del
c\|_{L^2}^2 dt  \le C, \]
with  $C=C(T, \|\chi(c)\|_{L^{\infty}}, \|\langle x \rangle
n_0\|_{L^1},  \| \na c_0\|_{L^2},
 \|{n_0}|\ln {n_0}|\|_{L^1},  \|\Del \phi\|_{L^{\infty}}, \|\na \phi\|_{L^{\infty}},
 \|\phi\|_{L^{\infty}})$.
By the same reasoning for treating $ n(\ln n)_-$ term in
\eqref{negative} we show the weak solutions $(n, c, u)$ satisfy the
energy inequality in Definition \ref{defweak} (c). This completes
the proof of Theorem $4$.
\end{pfthm4}
\end{subsection}

%=================================================
%=================================================
%=================================================
\section*{Acknowledgments}
M. Chae's work was supported by the National Research Foundation of
Korea(NRF No. 2009-0069501). K. Kang's work was partially supported by KRF-2008-331-C00024 and NRF-2009-0088692. J. Lee's work was partially supported by NRF-2009-0072320. We appreciate professor Dmitry Vorotnikov for valuable comments.

\begin{equation*}
\left.
\begin{array}{cc}
{\mbox{Myeongju Chae}}\qquad&\qquad {\mbox{Kyungkeun Kang}}\\
{\mbox{Department of Applied Mathematics }}\qquad&\qquad
 {\mbox{Department of Mathematics}} \\
{\mbox{Hankyong National University
}}\qquad&\qquad{\mbox{Yonsei University}}\\
{\mbox{Ansung, Republic of Korea}}\qquad&\qquad{\mbox{Seoul, Republic of Korea}}\\
{\mbox{mchae@hknu.ac.kr }}\qquad&\qquad {\mbox{kkang@yonsei.ac.kr }}
\end{array}\right.
\end{equation*}
\begin{equation*}
\left.
\begin{array}{c}
{\mbox{Jihoon Lee}}\\
{\mbox{Department of Mathematics }}\\
{\mbox{Sungkyunkwan University}}\\
{\mbox{Suwon, Republic of Korea}}\\
{\mbox{jihoonlee@skku.edu }}
\end{array}\right.
\end{equation*}

\end{document}